\documentclass[reqno, 11pt, a4paper]{amsart}

\usepackage{amsfonts, amsthm, amsmath, amssymb}

\usepackage[dvipsnames]{xcolor}
\usepackage[
    colorlinks=true,
    linkcolor=MidnightBlue,
    citecolor=MidnightBlue,
    urlcolor=MidnightBlue
]{hyperref}

\usepackage{booktabs}
\usepackage[font=small, margin=40pt,figureposition=below]{caption}
\usepackage{graphicx}
\usepackage{tikz-cd}
\usepackage{enumitem}
\usepackage{comment}
\usepackage[margin=3.1cm]{geometry}

\numberwithin{equation}{section}

\renewcommand{\d}{\mathrm{d}}

\renewcommand{\rho}{\varrho}
\newcommand{\0}{\mathbf{0}}
\newcommand{\PP}{\mathbb{P}}
\newcommand{\ZZ}{\mathbb{Z}}
\newcommand{\ZZp}{\mathbb{Z}_{\mathrm{prim}}}
\newcommand{\NN}{\mathbb{N}}
\newcommand{\QQ}{\mathbb{Q}}
\newcommand{\RR}{\mathbb{R}}

\renewcommand{\leq}{\leqslant}
\renewcommand{\le}{\leqslant}
\renewcommand{\geq}{\geqslant}
\newcommand{\h}{\mathbf{h}}
\newcommand{\q}{\mathbf{q}}
\newcommand{\x}{\mathbf{x}} 
\newcommand{\f}{\mathbf{f}}
\newcommand{\g}{\mathbf{g}}
\newcommand{\y}{\mathbf{y}}
\newcommand{\X}{\mathbf{X}}
\newcommand{\Y}{\mathbf{Y}}
\renewcommand{\c}{\mathbf{c}}
\newcommand{\z}{\mathbf{z}}
\renewcommand{\r}{\mathbf{r}}
\newcommand{\ve}{\varepsilon}
\newcommand{\uu}{\underline}
\newcommand{\abs}[1]{\left\lvert#1\right\rvert}
\newcommand{\eps}{\varepsilon}

\DeclareMathOperator{\rank}{rank}

\DeclareMathOperator{\vol}{vol}
\DeclareMathOperator{\meas}{vol}
\DeclareMathOperator{\prim}{prim}
\DeclareMathOperator{\adj}{adj}

\newtheorem{theorem}{Theorem}[section] 
\newtheorem{corollary}[theorem]{Corollary}
\newtheorem{lemma}[theorem]{Lemma}

\theoremstyle{definition}

\begin{document}

\makeatletter
\def\subjclassname{$2020$ Mathematics Subject Classification}
\makeatother

	\title[Linear growth for biprojective hypersurfaces of small dimension]{Towards linear growth for biprojective\\ hypersurfaces of small dimension}
	\author{Tim Browning}
	\address{ISTA\\
		Am Campus 1\\
		3400 Klosterneuburg\\
		Austria}
	
	\email{tdb@ist.ac.at}

	\author{Matteo Verzobio}
	\address{ISTA\\
		Am Campus 1\\
		3400 Klosterneuburg\\
		Austria}
    \address{KU Leuven \\ Celestijnenlaan 200B\\ 3001 Heverlee (Leuven)\\ Belgium}
	\email{matteo.verzobio@gmail.com}

	
	\date{}
	
	\begin{abstract}		After removing suitable thin subsets, 
		we count rational points of bounded height on some 
		 hypersurfaces of small bidegree in 
		$\mathbb{P}^3\times\mathbb{P}^3$ and in 
		$\mathbb{P}^3\times\mathbb{P}^2$. 
One of the ingredients is a new estimate for the number of rational points of bounded height on diagonal quadratic surfaces which improves as the coefficients grow.
	\end{abstract}

	\maketitle

	\thispagestyle{empty}

\setcounter{tocdepth}{1}
\tableofcontents

	\section{Introduction}

Let $c_1,\dots,c_4\in \ZZ$ be nonzero such that $\gcd(c_1,\dots,c_4)=1$. 
Let $Z$ be the hypersurface in $\mathbb{P}^3\times \mathbb{P}^3$ defined by
\begin{equation}\label{eq:eq}
c_1x_1^ay_1^b+c_2x_2^ay_2^b+c_3x_3^ay_3^b+c_4x_4^ay_4^b=0,
\end{equation}
for given integers $a,b\geq 1$. This variety is smooth if $a$ or $b$ is equal to $1$, but singular otherwise. 
Since $-K_Z=\mathcal{O}(4-a,4-b)|_Z$, we see that $Z$
 is Fano if 
$1\le a\le 3$ and $1\le b\le 3$, as we now assume. 
The associated anticanonical height $H:Z(\QQ)\to \RR_{>0}$
is defined to be 
$
H(x,y)=|\x|^{4-a}|\y|^{4-b},
$
if a point $(x,y)\in Z(\QQ)$ is represented by a vector
$(\x,\y)\in \ZZ_{\mathrm{prim}}^4\times \ZZ_{\mathrm{prim}}^4$, where $|\cdot|$ denotes the supremum norm on $\RR^4$.
Let
\begin{equation}\label{eq:duck}
N(U,B)=\#\bigl\{(x,y)\in U:H(x,y)\le B\bigr\},
\end{equation}
for any subset $U\subset Z(\QQ)$. 
A weak form of Manin's conjecture~\cite{manin} predicts that there should exist a thin subset $W\subset Z(\QQ)$ 
(in the sense of 
Serre~\cite[\S 3.1]{Ser08})
such that
\[
N(U,B)=O_{\ve,Z}(B^{1+\ve}),
\]
for any $\ve>0$,
where
 $U=Z(\QQ)\setminus W$.
In fact, if $Z$ is smooth 
then  $Z$ has Picard rank 2 and the Manin conjecture predicts that 
$N(U,B)\sim c_U B\log B$,
as $B\to\infty$, where $c_U$ is a constant whose value is predicted by   Peyre~\cite{Pey95}.
When $Z$ is singular we expect a similar asymptotic formula, but the power of $\log B$ can be larger.

When $a=b=1$, the 
Manin conjecture was settled by Spencer~\cite{Spencer}. He proved the conjecture for smooth diagonal hypersurfaces of bidegree $(1,1)$ in $\mathbb{P}^s\times\mathbb{P}^s$ for $s\geq 2$, using the circle method as the main tool. In doing so he needed to remove the thin set 
\begin{equation}\label{eq:thin1}
W_1=\left\{(x,y)\in Z(\QQ): \prod_{i=1}^4 x_i y_i =0\right\}.
\end{equation}
The cases $(a,b)\in \{(1,2), (2,1)\}$ were handled by
	Browning and Heath-Brown~\cite{BHB-quadbundle}. They achieved this by combining counting arguments that come from the geometry of numbers with the circle method. In this setting it was necessary to remove the thin set $W_1\cup W_2$,
	where $W_2$ is the set of 
	$(x,y)\in Z(\QQ)\setminus W_1$ such that  $\prod_{i=1}^4 c_i x_i$ is a rational square if $(a,b)=(1,2)$, or such that $\prod_{i=1}^4c_i y_i$ is a rational square if $(a,b)=(2,1)$.
	 The case $(a,b)=(2,2)$ was studied by Browning and Hu in~\cite{BrowningHu}, where an asymptotic formula is proved for smooth hypersurfaces of bidegree $(2,2)$ in at least 35 variables.
 In the diagonal setting, substantially stronger results are available. Blomer and Brüdern~\cite[Theorem 1.1]{BlomerBruedern} proved the Manin conjecture for a broad class of multihomogeneous diagonal equations. Specialising their theorem to diagonal hypersurfaces of bidegree $(2,2)$,
 one obtains the expected asymptotic formula for rational points of bounded anticanonical height as soon as the diagonal equation has at least five summands. In both cases, the thin set $W_1$ needs to be removed.

 For $(a,b)\in \{(1,3), (3,1)\}$, 
 Batyrev and Tschinkel~\cite{BatyrevTschinkel} studied $Z$ in 1996 and famously produced
 the first counterexample to an earlier formulation of Manin's
	conjecture in which it was only allowed to remove rational points contained in proper subvarieties of $Z$. 
	Their example exhibits a dominant generically finite morphism
 \[f:Y\to Z,\] whose image contributes a growth rate of order $B(\log B)^3$, which 
 exceeds the expected growth of order
 $B\log B$. Since the image $f(Y)$ is Zariski dense in $Z$, these excess rational points cannot be removed by deleting a proper closed subset, thereby demonstrating (for the first time) that the removal of 
 thin sets of Type~II is an essential feature of refined versions of the Manin conjecture.
Our first result provides upper bounds for $N(U,B)$ that do not depend on the coefficients of $Z$.

\begin{theorem}\label{thm:appl}
		Let $a,b\in\{1,2,3\}$ and let $\ve>0$. There exists 
		 a thin subset $W\subset Z(\QQ)$ such that 
		\[
		N(U,B)\ll_{\ve} B^{\gamma_{a,b}+\ve},
		\]
		where 
		$U=Z(\QQ)\setminus W$ and 
		the exponents $\gamma_{a,b}$ are listed in the following  table:
		\begin{table}[h!]
			\centering
			\begin{tabular}{c|ccc}
				$a \backslash b$ & $1$ & $2$ & $3$ \\\hline
				$1$ & $1$ & $1$ & $\tfrac{12}{7}$ \\
				$2$ & $1$ & $1$ & $\tfrac{44}{25}$ \\
				$3$ & $\tfrac{12}{7}$ & $\tfrac{44}{25}$ & $\tfrac{12}{7}$
			\end{tabular}
		\end{table}
	\end{theorem}

We emphasise that the implied constant only depends on $\ve$ and not on the 
individual coefficients $c_1,\dots,c_4$ in \eqref{eq:eq}. The construction of the deleted thin set $W$ will be carried out in Section~\ref{final}.
	 The main bottleneck in obtaining the expected exponent $\gamma_{a,b}=1$ is that we do not have sharp bounds for the number of rational points of bounded height on smooth cubic surfaces. The best known bound is due to Salberger~\cite{salcubic}, who has proved that 
 there are $O_{\ve}(B^{\frac{12}{7}+\ve})$
 rational points of bounded height on the open subset of a cubic surface 
 formed by deleting the 27 lines, and the implied constant only depends on $\ve$.
 Heath-Brown 
\cite{HBcubic} has
 proved that the exponent $12/7$ can be improved to $4/3$ under a certain {\em rank hypothesis}. Assuming this hypothesis would allow us to improve our bounds, giving $\gamma_{1,3}=\gamma_{3,1}=\gamma_{3,3}=4/3$.

 \medskip

 Our next result concerns the  fourfold $Z'\subset \mathbb{P}^3\times \mathbb{P}^2$ defined by
\[
\ell_1(\y)x_1^2+\ell_2(\y)x_2^2+\ell_3(\y)x_3^2+\ell_4(\y)x_4^2=0,
\]
where $\y=(y_1,y_2,y_3)$ and
 $\ell_1,\ell_2,\ell_3,\ell_4\in \ZZ[\y]$
 are   linear forms such that 
 any 
 three of them are linearly independent. Then $Z'$ is a  smooth Fano variety 
 and  the  associated anticanonical height $H:Z'(\QQ)\to \RR_{>0}$
is defined to be 
$
H(x,y)=|\x|^{2}|\y|^{2},
$
if a point $(x,y)\in Z'(\QQ)$ is represented by a vector
$(\x,\y)\in \ZZ_{\mathrm{prim}}^4\times \ZZ_{\mathrm{prim}}^3$.
Let
\begin{equation}\label{eq:eq'}
N(U',B)=\#\bigl\{(x,y)\in U':H(x,y)\le B\bigr\},
\end{equation}
for any subset $U'\subset Z'(\QQ)$. 
The Manin conjecture predicts
 that
there exists a thin subset of rational points $W'\subset Z'(\QQ)$ such that 
$N(U',B)\sim c_{U'} B\log B$, 
as $B\to\infty$, where
 $U'=Z'(\QQ)\setminus W'$ and $c_{U'}$ is a constant predicted by 
 Peyre \cite{Pey95}. 
 We aim to prove an estimate  that is as close as possible to the bound 
$N(U',B)=O_{\ve,Z'}(B^{1+\varepsilon})$, for any $\ve>0$.

\begin{theorem}\label{thm:aplllinear}
Let $\ve>0$. 
 There exists a thin subset $W'\subset Z'(\QQ)$ such that
 \[
 N(U',B)\ll_{\ve,Z'} B^{\frac{18}{17}+\ve},
 \]
  where
 $U'=Z'(\QQ)\setminus W'$.
\end{theorem}
 
 It would be straightforward to extend our argument to produce the same bound 
 for singular $Z'$ in which it is only assumed that the 
  coefficient matrix of 
  $\ell_1,\dots,\ell_4$ has rank $3$. 
 When  the linear forms $\ell_1,\dots,\ell_4$ depend on  only two variables in \eqref{eq:eq'}, the variety defines a threefold in $\PP^3\times \PP^1$, 
a case that has been resolved by Bonolis, Browning and Huang~\cite{BonolisBrowningHuang}, under certain technical assumptions on the linear forms. The case 
 in which  $\ell_1,\dots,\ell_4$ have four variables generalises  the special case
 $\ell_i(y_1,\dots,y_4)=y_i$, for $1\leq i\leq 4$, that was 
  handled by   Browning and Heath-Brown~\cite{BHB-quadbundle}.

 \medskip
 
	The main input in  Theorem~\ref{thm:aplllinear} is a new bound for the number of rational points of bounded height on
	diagonal quadratic surfaces in $\PP^3$, which improves as the coefficients of the form grow.		Let $A_1,\dots,A_4\in \ZZ$ be nonzero such that 
		$\gcd(A_1,\dots,A_4)=1$. Let 
	\[
	Q(\x)=A_1x_1^2+A_2x_2^2+A_3x_3^2+A_4x_4^2 ,
	\]
	where $\x=(x_1,x_2,x_3,x_4)$.
	Define
	\begin{equation}\label{eq:ham}
	N(Q,B)=\#\left\{\x\in\ZZ_{\prim}^4 : Q(\x)=0,\ |x_i|\le B\right\},
	\end{equation}
	and let $N^{0}(Q,B)$ be the same quantity with points lying on rational
	lines contained in the quadric $Q=0$ removed. Let
	\begin{equation}\label{eq:T}
	T=\gcd(A_1A_2,A_1A_3,A_1A_4,A_2A_3,A_2A_4,A_3A_4),
	\end{equation}
with $\Delta_Q=\abs{A_1A_2A_3A_4}$ and 
	$\|Q\|=\max_{1\leq i\leq 4}\abs{A_i}$.
	
	\begin{theorem}\label{thm:main}
	Let $\ve>0$. Then
		\[
		N^0(Q,B)\ll_{\ve} 
B^{\frac{4}{3}+\ve}+B^{\frac 53+\ve}\frac{T^{\frac{1}{2}}\|Q\|}{\Delta_Q^{\frac{1}{3}}}+B^{2+\ve}\frac{T^{\frac{1}{2}}\|Q\|^{\frac 12}}{\Delta_Q^{\frac{1}{3}}}.
		\]
	\end{theorem}

When $A_1\cdots A_4\neq \square$, there are no rational lines on the quadric surface $Q=0$, so that $N(Q,B)=N^0(Q,B)$ in this case. A slightly sharper (but more technical) upper bound is proved in Theorem~\ref{prop:main}, which is  a key ingredient in our proof of Theorem~\ref{thm:aplllinear}. 
When we come to apply it, the coefficients are all of order $B$ and $T$ is small, so that the term $B^{\frac43+\ve}$ dominates. This is the main bottleneck of our method.

Theorem~\ref{thm:main} is not the first upper bound of its type. 
It follows from work of Browning~\cite[Theorem 3]{BrowningDiagonal} that 
 \[
 N(Q,B)\ll_{\ve} B^{\frac 32+\ve}+\frac{B^{2+\ve}}{\Delta_Q^{\frac16}},
 \]
	under the additional assumption $\Delta_Q$ is square-free. 
Our bound applies to all nonsingular diagonal quadratic forms, without any square-freeness assumption on $\Delta_Q$. Moreover, when the discriminant is sufficiently large relative to the coefficients, the leading term is $B^{\frac 43+\varepsilon}$, which improves upon the $B^{\frac 32+\varepsilon}$ term appearing in Browning's estimate. 
Further work was done by Browning and Heath-Brown~\cite{BrowningHeathBrown}, who treated arbitrary nonsingular quadratic surfaces, with a focus on removing the factor $B^\varepsilon$ from the $B^{2+\ve}$-term. In a similar vein, there is also work of Comtat~\cite{comtat}, which in our setting proves
\begin{equation}\label{eq:comtat}
 N(Q,B)=\begin{cases} 
 O\left(B^2\right) & \text{ if $A_1\cdots A_4\neq \square$,}\\
 O\left(B^2\log B\right) & \text{ if $A_1\cdots A_4= \square$,}
\end{cases}
\end{equation}
for an absolute implied constant.
 Finally, 
in the affine setting, 
 we note that the case of diagonal quadratic surfaces was also studied by the authors~\cite[Corollary 1.3]{bv-quad} using a modified form of the determinant method. 

\subsection*{Outline of the paper}
	
	To prove Theorem~\ref{thm:main}, we will employ a similar strategy 
	to one adopted in~\cite{BrowningDiagonal,BrowningHeathBrown}, exploiting the fact that we are dealing with a diagonal quadratic form. 
	 A point of bounded height must lie on a plane of
	height $O(B^{\frac 13})$, by Siegel's lemma. We intersect $Q=0$ with all planes of height $O(B^{\frac 13})$ and count the
	points on each of these intersections. The intersection of $Q=0$ with a plane is a conic, and
	we bound the number of rational points on each of these conics. This is carried out in Sections~\ref{sec:quad}--\ref{sec:quad''}.
	
We prove Theorem~\ref{thm:appl} in Section~\ref{final}, as follows. To bound the number of points of bounded height on $U$, we use three different methods. First, we fix $\x=(x_1,x_2,x_3,x_4)$ and we count points on the fibre $Z_{\x}\subset\mathbb{P}^3$
	with equation \eqref{eq:eq}, which defines a surface of degree $b$. 
If $b=3$, we employ the determinant method. More specifically, we use a version of the method introduced by
Ellenberg and Venkatesh~\cite{ev} and refined by Walsh~\cite{walsh}, which becomes more powerful as the coefficients of the form defining the variety grow. We then sum over all possible choices of $\x$. Second, we can instead count points on the surface $Z_{\y}\subset\mathbb{P}^3$ of degree $a$ in which we consider $\y$ as fixed and then sum over all possible choices of $\y$. Third, we employ a Fourier-analytic argument inspired by the proof of Hua's inequality in the circle method, following the spirit of recent work on exceptional sets related to the $abc$ conjecture~\cite[Section 3]{abc}. We then combine these three methods to obtain the best possible bounds.

Finally, we prove Theorem~\ref{thm:aplllinear} in Section~\ref{sec:lin}. To bound the number of points of bounded height on $U'$, we use two different methods. First, we fix $\x=(x_1,x_2,x_3,x_4)$ and we count points on the fibre $Z_{\x}'\subset\mathbb{P}^2$, which is a line. Second, we fix $\y=(y_1,y_2,y_3)$ and we count points on the fibre $Z_{\y}'\subset\mathbb{P}^3$, which is a diagonal quadratic surface. To do so, we use Theorem~\ref{thm:main} (and its more technical variant Theorem \ref{prop:main}). We then combine these two methods to obtain the best possible bound.

\subsection*{Acknowledgments} While working on this paper, the second author was supported by the European Union’s Horizon 2020 Research and Innovation Programme under the Marie Skłodowska-Curie Grant Agreement No. 101034413 at ISTA and by the Research Foundation – Flanders (FWO) through grants G0B3123N, G0B1721N, and G0A2923N at KU Leuven.

\subsection*{Notation}
Let $\QQ^{\times m}$ denote the group of non-zero rational $m$th powers, for any integer $m\geq 2$.
Throughout the paper, given $A,B>0$, we write $A\sim B$ to mean $B\leq A< 2B$. Given a positive integer $x$, we denote by $x^{\square}$ the largest square dividing $x$.

	\section{Uniform bounds for quadrics: preliminaries}\label{sec:quad}

Let $Q(\x)=A_1x_1^2+A_2x_2^2+A_3x_3^2+A_4x_4^2$ for relatively prime nonzero 
$A_1,\dots,A_4\in \ZZ$.
	We write 
	\[
	\Delta_Q=\abs{A_1A_2A_3A_4}, \quad 
	\|Q\|=\max_{1\leq i\leq 4}\abs{A_i}, \quad 
	m_Q=\min_{1\leq i\leq 4}\abs{A_i},
	\]
and we recall the definition \eqref{eq:T} of $T$.	
	
	 Let $\c=(c_1,c_2,c_3,c_4)\in\ZZ_{\prim}^4$.
	Assume that $\abs{c_4}=\max_{1\leq i\leq 4}\abs{c_i}=\abs{\c}$. A point that satisfies $Q(\x)=\x\cdot\c=0$ must satisfy $q_{\c}(x_1,x_2,x_3)=0$, where
	\begin{align*}
		q_{\c}(x_1,x_2,x_3)=~&c_4^2\left(A_1x_1^2+A_2x_2^2+A_3x_3^2+A_4\frac{(c_1x_1+c_2x_2+c_3x_3)^2}{c_4^2}\right)\\=~&c_4^2A_1x_1^2+c_1^2A_4x_1^2+c_4^2A_2x_2^2+c_2^2A_4x_2^2+c_4^2A_3x_3^2+c_3^2A_4x_3^2\\
		&+2A_4c_1c_2x_1x_2+2A_4c_2c_3x_2x_3+2A_4c_1c_3x_1x_3.
	\end{align*}
	The associated symmetric matrix is
	\begin{equation}\label{eq:Mc}
		M_{\c}=\begin{pmatrix}
			c_4^2A_1+c_1^2A_4& A_4c_1c_2& A_4c_1c_3 \\
			A_4c_1c_2& c_4^2A_2+c_2^2A_4& A_4c_2c_3 \\
			A_4c_1c_3& A_4c_2c_3& c_4^2A_3+c_3^2A_4 \\
		\end{pmatrix}.
	\end{equation}
We have $\det M_{\c}=c_4^4Q^*(\c)$, where
	\begin{equation}\label{eq:Q*}
		Q^*(\c)=A_1A_2A_3c_4^2+A_1A_2A_4c_3^2+A_1A_3A_4c_2^2+A_2A_3A_4c_1^2
	\end{equation}
is the dual form to $Q$. 	

	The
	adjugate matrix
	 $D_{\c}$ of $M_{\c}$ is equal to 
	\begin{equation}\label{eq:defDc}
\begin{footnotesize}
		c_4^2\begin{pmatrix}
			A_2A_3c_4^2+A_2A_4c_3^2+A_3A_4c_2^2 	& -c_1c_2A_3A_4& -c_1c_3A_2A_4 \\
			-c_1c_2A_3A_4& 	A_1A_3c_4^2+A_3A_4c_1^2+A_1A_4c_3^2 &-c_2c_3A_1A_4 \\
			-	c_1c_3A_2A_4& -c_2c_3A_1A_4& A_1A_2c_4^2+A_2A_4c_1^2+A_1A_4c_2^2\\
		\end{pmatrix}.
\end{footnotesize}
	\end{equation}
We let $d_\c$ 
	be the greatest common divisor of the entries of $D_\c$. 
		Let $g_{\c}=\gcd(A_4,c_4^2)$ and $A_4'=A_4/g_{\c}$. Let
	\begin{equation}\label{eq:Tcc}
	T_{\c}=\gcd(A_1A_2,A_1A_3,A_1A_4',A_2A_3,A_2A_4',A_3A_4')
	\end{equation}
and let 
	\[
	D_\c'=\frac{D_\c}{T_{\c}c_4^2g_{\c}},
	\] 
noting that $D_\c'$ has integer entries. 
	
	\begin{lemma}\label{lem:holder}
		Let $R,X_1,\dots, X_n\geq 1$, let $B_1,\dots, B_n\in \ZZ_{\neq 0}$ and let $\ve>0$. Let
		\begin{equation}\label{eq:NRu}
		N_R(\uu{X})=\#\left\{\x\in (\ZZ_{\neq 0})^n:\abs{\sum_{i=1}^n B_ix_i^2}\leq R, \, \abs{x_i}\leq X_i\right\}
		\end{equation}
		and put $X=\max_{1\leq i\leq n}X_i$.
		For $n\geq 4$, we have
		\[
		N_R(\uu{X})\ll_{n,\ve} (RX)^\ve\left(\prod_{i=1}^n X_i\right)^{\frac{n-2}{n}}\left(\prod_{i=1}^n \left(1+\frac{R}{\abs{B_i}}\right)\right)^{\frac{1}{n}}.
		\]
		For $n=3$, we have
		\[
		N_R(\uu{X})\ll_\ve (RX)^\ve \prod_{i=1}^3
\left(
		\left(
1+\frac{R}{|B_i|}\right)^{\frac{1}{4}}
X_i^{\frac{1}{2}}\right).
				\]
	\end{lemma}
	\begin{proof}
		Let 
		\[
		M_i(m)=\left\{x_i\in \ZZ_{\neq 0}: \abs{x_i}\leq X_i, (m-1)R\leq B_ix_i^2\leq mR\right\}
		\]
		and $a_{i,m}=\#M_i(m)$. Notice that if $-R\leq \sum_{i=1}^n B_ix_i^2\leq R$ and $x_i\in M_i(m_i)$ for certain $m_i$, then
		\[
		-R+\sum_{i=1}^n (m_i-1)R\leq 0\leq R+\sum_{i=1}^n m_iR
		\] 
		and thus
		\[
		-1\leq \sum_{i=1}^n m_i\leq n+1.
		\]
		It follows that
		\begin{align*}
		N_R(\uu{X})
		&\leq \sum_{-1\leq \sum_i m_i\leq n+1}\prod_{i=1}^n a_{i,m_i}\\ 
		&= \sum_{-1\leq k\ll_n 1}\int_0^1 \sum_{m_i\in \ZZ} \left( \prod_{i=1}^n a_{i,m_i}\right)e\left(\alpha\left(\sum_i m_i-k\right) \right) \d\alpha,
		\end{align*}
		where we denote $\exp(2\pi i x)$ by $e(x)$. Notice that $a_{i,m}=0$ for $\abs{m}> \abs{B_iX_i^2/R}+1$ and so the sum over the $m_i$ has finitely many nonzero terms.
		Let
		\[
		S_i(\alpha)=\sum_{m\in \ZZ} a_{i,m}e(\alpha m ),
		\]
		so that
		\[
		N_R(\uu{X})\ll\sum_{-1\leq k\ll_n 1}\int_0^1 e(-\alpha k)\prod_{i=1}^n S_i(\alpha) \d\alpha.
		\]
		H\"older's inequality yields
		\[
		N_R(\uu{X})\ll_n \int_0^1 \prod_{i=1}^n \abs{S_i(\alpha)} \d\alpha\leq \left(\prod_{i=1}^n\int_0^1 \abs{S_i(\alpha)}^n \d\alpha\right)^{\frac 1n}.
		\]
		
		Assume first that $n\geq 4$ and notice that $\abs{S_i(\alpha)}\leq \sum_{m\in \ZZ} a_{i,m}\ll X_i$. Thus
		\begin{equation}\label{eq:NRX}
			N_R(\uu{X})\ll_n \left(\prod_{i=1}^n X_i^{n-4}\int_0^1 \abs{S_i(\alpha)}^4 \d\alpha\right)^{\frac 1n}.
		\end{equation}
		We have
		\[
		\int_0^1 \abs{S_i(\alpha)}^4 \d\alpha=\sum_{m_1+m_2+m_3+m_4=0} a_{i,m_1}a_{i,-m_2}a_{i,m_3}a_{i,-m_4} \leq \#M_i,
		\]
where
		\[
		M_i=\left\{\x\in \ZZ_{\neq 0}^4: \abs{x_j}\leq X_i \text{ for }1\leq j\leq 4, \abs{B_i\left(x_1^2-x_2^2+x_3^2-x_4^2\right)}\leq 2R\right\}.
		\]
				Letting
		\[
		M_i(N)=\left\{(x_1,x_2,x_3,x_4)\in \ZZ_{\neq 0}^4: \abs{x_j}\leq X_i \text{ for }1\leq j\leq 4, ~x_1^2-x_2^2+x_3^2-x_4^2=N\right\},
		\]
we may write
		\[
		M_i=\bigcup_{\abs{N}\leq \left\lfloor \frac{2R}{\abs{B_i}} \right\rfloor}M_i(N).
		\]
		We have 
$\#M_i(N)\ll_{\ve} (1+\abs{N})^{\ve}X_i^{2+\ve}$
		for all $N$, whence
		\begin{equation}\label{eq:toad}
				\int_0^1 \abs{S_i(\alpha)}^4 \d\alpha\leq 
		\#M_i\ll_\ve X_i^{2+\ve}\left(1+\frac{R}{\abs{B_i}}\right)^{1+\ve}.
		\end{equation}
On appealing to \eqref{eq:NRX}, we conclude that 
		\[
		N_R(\uu{X})\ll_{n,\ve}\left(\prod_{i=1}^n X_i^{n-4}\cdot X_i^{2+\ve}\left(1+\frac{R}{\abs{B_i}}\right)^{1+\ve}\right)^{\frac 1n},
		\]
		which thereby completes the case $n\geq 4$. 

Suppose next that $n=3$. 
Then 
		\[
		\int_0^1 \abs{S_i(\alpha)}^2 \d\alpha\leq \#M_i',
		\]
where
		\[
		M_i'=\left\{(x_1,x_2)\in \ZZ_{\neq 0}^2: \abs{x_j}\leq X_i \text{ for }1\leq j\leq 2, \abs{B_i\left(x_1^2-x_2^2\right)}\leq R\right\}.
		\]
    Notice that $\#M_i'\ll X_i^2$. Moreover, on
		letting
		\[
		M_i'(N)=\left\{(x_1,x_2)\in \ZZ_{\neq 0}^2: \abs{x_j}\leq X_i \text{ for }1\leq j\leq 2, ~x_1^2-x_2^2=N\right\},
		\]	
		for any $N\in \ZZ$, 
		we see that $\#M_i'(0)\ll X_i$ and 
		$
		\#M_i'(N) \ll_\ve |N|^\ve$ if $N\neq 0$, by the divisor bound.
Hence
 it follows that
		\begin{align*}
			\#M_i'
			&\ll_\ve X_i
			+ R^\ve \min\left\{\frac{R}{\abs{B_i}}, 
		X_i^2\right\},
					\end{align*}
		Finally, on appealing to Cauchy--Schwarz and combining the previous estimate with \eqref{eq:toad}, we deduce that 
		\begin{align*}
			\int_0^1 \abs{S_i(\alpha)}^3 \d\alpha
			&\leq \left(\int_0^1 \abs{S_i(\alpha)}^2 \d\alpha\right)^{\frac 12}\left(\int_0^1 \abs{S_i(\alpha)}^4 \d\alpha\right)^{\frac 12} 
			\\
			&\ll_{\ve} 
			X_i^{1+\ve}\left(1+\frac{R}{\abs{B_i}}\right)^{\frac{1}{2}+\ve} \cdot 
			\left(X_i
			+ R^\ve \min\left\{\frac{R}{\abs{B_i}}, 
		X_i^2\right\}\right)^\frac{1}{2}\\
			&\ll_{\ve} 
			\left(1+\frac{R}{\abs{B_i}}\right)^{\frac{1}{2}+\ve}X_i^{\frac{3}{2}+\ve} +
						 (X_iR)^\ve \left(
			X_i^2\left(1+\frac{R}{\abs{B_i}}\right) 
						\min\left\{\frac{R}{\abs{B_i}}, 
		X_i^2\right\}\right)^\frac{1}{2}.
		\end{align*}
		If $R/|B_i|\ll 1$ then the overall contribution is $O_\ve (R^\ve X_i^{\frac{3}{2}+\ve})$.
		Thus we may suppose that $R/|B_i|\gg 1$. Taking 
\[		\min\left\{\frac{R}{\abs{B_i}}, 
		X_i^2\right\}\leq 
		\left(\frac{R}{\abs{B_i}}\right)^{\frac{1}{2}} X_i,
\]		
we conclude that 
		\begin{align*}
			\int_0^1 \abs{S_i(\alpha)}^3 \d\alpha
&\ll_{\ve} (RX_i)^\ve \left(
1+\frac{R}{|B_i|}\right)^{\frac{3}{4}}
X_i^{\frac{3}{2}} ,\end{align*}	
from which the statement of the lemma follows when $n=3$.
	\end{proof}

	\begin{corollary}\label{lem:NRX}
		Let $R,X_1,\dots, X_n\geq 1$, let $B_1,\dots, B_n \in \ZZ_{\neq 0}$ and let $\ve>0.$ 
Let $N_R(\uu{X})$ be given by \eqref{eq:NRu} and let $j\in \{1,\dots,n\}$ be such that $\abs{B_j}=\min_{1\leq i\leq n}\{\abs{B_i}\}$.
For $n\geq 4$, we have
		\[
		N_R(\uu{X})\ll_{n,\ve} \left(\prod_{i=1}^n X_i\right)\left(\frac{R}{X_j^2\abs{B_j}}\right)^{\frac 12}+R^{\ve}\left(\prod_{i=1}^n X_i\right)^{\frac{n-2}{n}+\ve}.
		\]
		For $n=3$, we have
		\[
		N_R(\uu{X})\ll_{\ve} \left(\prod_{i=1}^3 X_i\right)\left(\frac{R}{X_j^2\abs{B_j}}\right)^{\frac 12}+R^{\ve}\left(\prod_{i=1}^3 X_i\right)^{\frac{1}{2}+\ve}.
		\]
	\end{corollary}
	\begin{proof}
		Assume without loss of generality that $j=1$. 
		If 
 $R<\abs{B_1}$, the desired bound follows immediately from Lemma~\ref{lem:holder}.
				Assume that $ R\geq \abs{B_1}$. We have $O_n(\prod_{i=2}^n X_i)$ choices for $x_2,x_3,\dots, x_n$. Once fixed, it follows that $B_1x_1^2$ must lie in an interval of length $R$, and so $\abs{x_1}$ lies in an interval of length $O(\sqrt{R/\abs{B_1}})$. Thus, we have $O(\sqrt{R/\abs{B_1}})$ choices for $x_1$, leading to a satisfactory overall estimate.
	\end{proof}
	
	For certain ranges of $R,X_i,B_i$, the bound given in Lemma~\ref{lem:holder} is significantly better than the one in Corollary~\ref{lem:NRX}, but for our purposes it is more convenient to use the latter. 

	\section{Uniform bounds for quadrics: the main estimates}\label{sec:quad'}
	
For $\c\in \ZZ_{\prim}^4$, the main goal of this section is to estimate the cardinality of rational points of bounded height that lie on the intersection of $Q(\x)=0$ with a plane $\c\cdot \x=0$, before summing this contribution over vectors 
$\c\in \ZZ_{\prim}^4$	 such that $|\c|$ and $|Q^*(\c)|$ are constrained to lie in dyadic intervals. 

		\subsection{No vanishing coordinates}\label{sec:no}
We begin by focussing on the case $c_1\dots c_4\neq 0$.
		Recall the notation for $g_\c$ that was introduced at the start of Section~\ref{sec:quad}.
Let $d_\c'$ 
be the greatest common divisor of the entries of $D_\c'$.
		Let $C,g,R,h\in \ZZ_{\geq 1}$. We study the set
	\begin{equation}\label{eq:defM1}
		M_1(C,g,R,h)=\left\{\c\in \ZZ_{\prim}^4:
		\begin{array}{l}
		\abs{\c}=\abs{c_4}\sim C,~g_{\c}=g, ~d_\c'=h, \\
		\abs{Q^*(\c)}\sim R,~c_1\cdots c_4\neq 0
		\end{array}
		\right\},
	\end{equation}
with the aim of computing a uniform upper bound for its cardinality. We will address this in Lemma~\ref{lem:cardM}. For now, we prove a preliminary lemma.
	
	\begin{lemma}\label{lem:cimidhi}
		If $\c\in M_1(C,g,R,h)$, then there exist $h_1,h_2,h_3$ such that $h=h_1h_2h_3$ and $h_i\mid c_i$ for $i=1,2,3$. Moreover, we have $(g/\sqrt{g^{\square}})\mid c_4$.
	\end{lemma}
	
	\begin{proof}
	Write $g=g_{\c}$ and $A_4'=A_4/g$. 
			The second part of the lemma follows on noting that $g\mid c_4^2$ implies $(g/\sqrt{g^{\square}})\mid c_4$. Turning to the first part we begin by proving that 
	\begin{equation}\label{eq:ram}
		h\mid c_1c_2c_3,
	\end{equation}
for any 
	$\c\in M_1(C,g,R,h)$, where $h=d_{\c}'$.
This will suffice for the lemma, since we may then take
		\[
		h_1=\gcd(h,c_1),\qquad
		h_2=\gcd\left(\frac{h}{h_1},c_2\right),\qquad
		h_3=\frac{h}{h_1h_2}. 
	\]
To prove \eqref{eq:ram}, we fix a prime $p\mid h$ and put
	\[
		\rho=v_p(h),\qquad
		\alpha_i=v_p(A_i)\quad (1\leq i\leq 3),\qquad
		\alpha_4=v_p(A_4'),\qquad
		\tau=v_p(T_{\c}).
	\]
	By the definition of $T_{\c}$,
	\[
		\tau=\min_{1\leq r<s\leq 4}(\alpha_r+\alpha_s).
	\]
	We claim that there exists 
	$k\in\{1,2,3\}$ such that
$		\alpha_k+\alpha_4=\tau$.
	Suppose for a contradiction that 
$		\alpha_k+\alpha_4>\tau$
for $1\leq k\leq 3$. Then, in particular, it follows that $\alpha_4>0$.
Since $A_4'=A_4/g$ and
	$g=\gcd(A_4,c_4^2)$, the inequality $\alpha_4>0$ implies
that $v_p(g)=v_p(c_4^2)$.
	Moreover, there must exist distinct $i,j\in\{1,2,3\}$
	such that $\alpha_i+\alpha_j=\tau$.
If $k\in \{1,2,3\}$ is the remaining index, it follows from \eqref{eq:defDc} that 
	\[
		(D_{\c}')_{kk}
		=
		\frac{
			A_iA_jc_4^2+
			A_iA_4c_j^2+
			A_jA_4c_i^2
		}{gT_{\c}}.
	\]
	The first summand on the right has $p$-adic valuation
	$
		\alpha_i+\alpha_j+v_p(c_4^2)-v_p(g)-\tau=0$,
	whereas the other two have valuations
	\[
		\alpha_i+\alpha_4+2v_p(c_j)-\tau>0
	\]
	and
	\[
		\alpha_j+\alpha_4+2v_p(c_i)-\tau>0.
	\]
But then 
		$v_p\bigl((D_{\c}')_{kk}\bigr)=0$,
which 	contradicts the inequality $\rho>0$, since $h$ divides every entry of
	$D_{\c}'$. This proves the claim.

	Choose $k\in \{1,2,3\}$ 
	such that $\alpha_k+\alpha_4=\tau$ and let 
	$\{i,j,k\}=\{1,2,3\}$. The corresponding off-diagonal entry is
	\[
		(D_{\c}')_{ij}
		=-c_ic_j\frac{A_kA_4'}{T_{\c}}.
	\]
The right hand side has $p$-adic valuation $v_p(c_ic_j)+\alpha_k+\alpha_4-\tau=v_p(c_ic_j)$. 
Since $p^\rho$ divides every entry of $D_{\c}'$, we therefore obtain
$\rho\leq v_p(c_i)+v_p(c_j)$. 
	Thus $\rho\leq v_p(c_1c_2c_3)$ for every prime $p$, which thereby proves \eqref{eq:ram}.
\end{proof}
	
	We are now ready to record our estimate for the cardinality of the set $M_1(C,g,R,h)$ that was defined in~\eqref{eq:defM1}.
	
	\begin{lemma}\label{lem:cardM}
		Let $C,g,R,h\in \ZZ_{\geq 1}$. For all $\ve>0$, we have
		\[
		\#	M_1(C,g,R,h)\ll_\ve (RC)^{\ve}\left(\frac{C^3R^{\frac{1}{2}} \|Q\|^{\frac{1}{2}}}{g^{\frac{1}{2}} \Delta_Q^{\frac{1}{2}}h}+\frac{C^{2}(g^{\square})^{\frac{1}{4}}}{h^{\frac 12}g^{\frac 12}}\right).
		\]
	\end{lemma}
	
	\begin{proof}
		Given $h_1,h_2,h_3\in \ZZ_{\geq 1}$, let 
		$M_1(C,g,R,h_1,h_2,h_3)$ be the set of $\c\in M_1(C,g,R,h)$ such that $h_i\mid c_i$, for $1\leq i\leq 3$. Then it follows from 
		 Lemma~\ref{lem:cimidhi} that 
		\[
		\#M_1(C,g,R,h)\leq \sum_{h_1h_2h_3=h} \#M_1(C,g,R,h_1,h_2,h_3).
		\]
		Since there are $O_{\ve}(C^\ve)$ choices for $h_1,h_2,h_3$ such that $h_1h_2h_3=h$, we just need to bound the cardinality of $M_1(C,g,R,h_1,h_2,h_3)$. 
		Set $h_4=g/\sqrt{g^{\square}}$. 
		It follows from Lemma~\ref{lem:cimidhi} that $h_4\mid c_4$. Hence we can define integers $c_k'=c_k/h_k$, for $1\leq k\leq 4$.
		Recall that $Q^*$ is defined in~\eqref{eq:Q*}. Putting $B_i=A_1A_2A_3A_4h_i^2/A_i$, we see that 
		\[
		Q^*(\c)=B_1c_1'^2+B_2c_2'^2+B_3c_3'^2+B_4c_4'^2,
		\]
whence
		\[
		\#M_1(C,g,R,h_1,h_2,h_3)\leq \#\left\{{\c'}\in(\ZZ_{\neq 0})^4:\abs{\sum_{i=1}^4 B_ic_i'^2}\leq 2R, |c_i'|\leq \frac{2C}{h_i}\right\}.
		\]
		Let $j$ be such that $\abs{B_j}=\min_i\{\abs{B_i}\}$. Then it follows from 
Corollary~\ref{lem:NRX} that 
		\begin{align*}
			\#M_1(C,g,R,h_1,h_2,h_3)&\ll_\ve 
			\frac{C^{2+\ve}R^{\ve}(g^{\square})^{\frac{1}{4}}}{h^{\frac 12}g^{\frac 12}}
			+\frac{C^{4}(g^{\square})^{\frac{1}{2}}}{hg}\left(\frac{Rh_j^2}{|B_j|C^2}\right)^{\frac 12}\\
			&\ll_\ve 
			\frac{C^{2+\ve}R^{\ve}(g^{\square})^{\frac{1}{4}}}{h^{\frac 12}g^{\frac 12}}
			+\frac{C^3}{hg^{\frac 12}} \left(\frac{R\|Q\|}{\Delta_Q}\right)^{\frac 12},
		\end{align*}
		which is satisfactory.
	\end{proof}
	
	Our goal now is to prove an upper bound for the number of rational points of bounded height on the curve $q_\c$, as defined at the beginning of the section. We will eventually achieve this in Lemma~\ref{lem:countM1}, after some preliminary technical results.

	\begin{lemma}\label{lem:ellips}
		Let $\c=(c_1,c_2,c_3,c_4)\in\ZZ_{\prim}^4$ with $\abs{c_4}=\abs{\c}$ and $Q^*(\c)\neq 0$.
		There are ellipsoids $\{E_j\}_{1\leq j\leq J}$, with $J=O_{\ve}((\abs{\c}\|Q\|)^\ve)$, each centred at the origin and of volume
		\[
		O\left(	B^3\min\left(1,\frac{\abs{Q^*(\c)}\|Q\|^3}{c_4^2\Delta_Q^{\frac{3}{2}}}\right)\right),
		\]
		such that
		\[
		\{\x\in\RR^3:q_\c(\x)=0,\abs{\x}\leq B\}\subset \bigcup_{j=1}^J E_j.
		\]
	\end{lemma}
	
	\begin{proof}
		Let $z=(c_1x_1+c_2x_2+c_3x_3+c_4x_4)/c_4$. We have
		\begin{align*}
			c_4^2Q(\x)=&A_1c_4^2x_1^2+A_2c_4^2x_2^2+A_3c_4^2x_3^2+A_4(c_4x_4)^2
			\\=&A_1c_4^2x_1^2+A_2c_4^2x_2^2+A_3c_4^2x_3^2+A_4(c_4z-c_1x_1-c_2x_2-c_3x_3)^2\\=&
			q_{\c}(\x)-2A_4c_1c_4x_1z-2A_4c_2c_4x_2z-2A_4c_3c_4x_3z+A_4c_4^2z^2.
		\end{align*}
		Let
		\[
		f_{\c}(x_1,x_2,x_3,z)=\frac{q_{\c}(\x)-2A_4c_1c_4x_1z-2A_4c_2c_4x_2z-2A_4c_3c_4x_3z+A_4c_4^2z^2}{c_4^2},
		\]
		and let $N_{\c}$ be the associated $4\times 4$ matrix, so that $\abs{\det(N_{\c})}=\Delta_Q$. Notice that the $3\times 3$ upper left sub-matrix of $N_{\c}$ is $M_{\c}/c_4^2$, where $M_{\c}$ is defined in~\eqref{eq:Mc}. Let $U'$ be the orthogonal matrix that diagonalises the symmetric matrix $M_{\c}/c_4^2$, and let 
		\[
		U=\begin{pmatrix}
			U' & 0\\
			0&1\\
		\end{pmatrix}.
		\]
Then 
		\[
		U^TN_{\c}U=\begin{pmatrix}
			\mu_1 & 0&0&*\\
			0 & \mu_2&0&*\\
			0 & 0&\mu_3&*\\
			* & *&*&*\\
		\end{pmatrix}.
		\]
		Notice that the determinant of the matrix is $\det(N_\c)$ and all the entries are bounded by $O(\|Q\|)$ (since the entries of an orthogonal matrix have absolute value bounded by $1$). If $\abs{\mu_3}\leq\abs{\mu_2}\leq \abs{\mu_1}$, then the determinant of the matrix is $O(\abs{\mu_1}\abs{\mu_2}\|Q\|^2)$ and so $\abs{\mu_1\mu_2}\gg \Delta_Q/\|Q\|^2$. Moreover, 
		\[
		\mu_1\mu_2\mu_3=\frac{\det(M_{\c})}{c_4^6}=\frac{Q^*(\c)}{c_4^2},
		\] 
from which it follows that 
		\[
		\frac{\abs{\mu_3}}{\abs{\mu_1\mu_2}^{\frac{1}{2}}}=\frac{\abs{\mu_3\mu_2\mu_1}}{\abs{\mu_1\mu_2}^{\frac 32}}\ll \frac{\abs{Q^*(\c)}\|Q\|^3}{c_4^2\Delta_Q^{\frac{3}{2}}}.
		\]
		We now appeal to the second half of the proof of~\cite[Lemma 2.2]{BrowningHeathBrown}.
Given $f(z_1,z_2,z_3)=\mu_1z_1^2+\mu_2z_2^2+\mu_3z_3^2$, this produces covering ellipsoids $\{E_j\}_{1\leq j\leq J}$, with $J=O_{\ve}(\|f\|^\ve)$, which are centred at the origin and of measure
		\[
		O\left(	B^3\min\left(1,\frac{\abs{\mu_3}}{\abs{\mu_1\mu_2}^{\frac{1}{2}}}\right)\right).
		\]
In our setting this implies that the solutions to $q_\c(\x)=0$
of bounded height 
 lie on $O_{\ve}((\abs{\c}\|Q\|)^{\ve})$ ellipsoids, each of volume
		\[
\ll B^3\min\left(1,\frac{\abs{\mu_3}}{\abs{\mu_1\mu_2}^{\frac{1}{2}}}\right) \ll 	B^3\min\left(1,\frac{\abs{Q^*(\c)}\|Q\|^3}{c_4^2\Delta_Q^{\frac{3}{2}}}\right),
		\]
which thereby completes the proof. 
	\end{proof}
	
	Recall that $d_{\c}'$ is defined to be the greatest common divisor of the entries of $D_\c'$.
	The following result allows us to cover the integer solutions $q_\c=0$ by a small number of lattices of large discriminant. 
	
	\begin{lemma}\label{lem:lattices}
		Let $\c=(c_1,c_2,c_3,c_4)\in\ZZ_{\prim}^4$ with $\abs{c_4}=\abs{\c}$ and $Q^*(\c)\neq 0$.
		There exist lattices $\{\Lambda_i\}_{1\leq i\leq I}$, with $I=O_{\ve}((\abs{\c}\|Q\|)^\ve)$, such that $\Lambda_i\subset \ZZ^3$ and 
		\[
		\det \Lambda_i \gg \frac{\abs{c_4Q^*(\c)}}{\abs{g_{\c}Td_\c'}^{\frac{3}{2}}},
		\]
		for $1\leq i\leq I$, and 
		such that 
		\[
		\left\{\x\in \ZZ_{\prim}^3: q_\c(\x)=0\right\}\subset \bigcup_{i=1}^I\Lambda_i.
		\]
	\end{lemma}
	
	\begin{proof}
	This is an easy consequence of~\cite[Lemma 2.6]{BrowningHeathBrown}, 
	which shows that 
	the solutions of $q_\c=0$ lie on $I=O((\abs{\c}\|Q\|)^\ve)$ integer lattices of discriminant 
		\[
		\gg \frac{\abs{\det(M_{\c})}}{\abs{d_\c}^{\frac{3}{2}}}.
		\]
		We conclude by observing that 
		$\det(M_{\c})=c_4^4 Q^*(\c) $ and 
		$d_\c=d_\c'T_{\c}c_4^2g_\c$ and $T_{\c}\leq T$.
	\end{proof}
	
	\begin{lemma}\label{lem:latell}
		Let $\Lambda$ be a lattice and $E$ be an ellipsoid centred at the origin in $\RR^3$. There exists a basis $\f_1,\f_2,\f_3$ of $\Lambda$ and positive numbers $L_1,L_2,L_3$ satisfying 
		\[
		L_1L_2L_3\ll\frac{\meas{E}}{\det(\Lambda)},
		\]
such that any $\x\in E\cap \Lambda$ can be written $\x=y_1\f_1+y_2 \f_2+y_3 \f_3$ for $y_i\in \ZZ$ satisfying $\abs{y_i}\leq L_i$, for $i=1,2,3$.
	\end{lemma}
	
	\begin{proof}
		Let $A\in \operatorname{GL}_3(\RR)$ send $E$ to the unit ball, so that 
		$\det(A)\meas{E}=\vol(B_3)\asymp 1$. Let $\Lambda'=A\Lambda$ and $\g=(\g_1,\g_2,\g_3)$ be a basis of $\Lambda'$ as in~\cite[Lemma 2.3]{BrowningHeathBrown}. 
		If $\z=\sum_{1\leq i\leq 3} y_i \g_i\in \Lambda'$ and $\z$ belongs to the unit cube, then $\abs{y_i}\ll 1/\abs{\g_i}$. 
		Let $\f=A^{-1}\g$, which is a basis of $\Lambda$. If $\x=\sum_{1\leq i\leq 3} y_i \f_i\in \Lambda$ and $\x$ belongs to $E$, then $\z=A\x$ is such that $\z=\sum_{1\leq i\leq 3} y_i \g_i\in \Lambda'$ and $\z$ belongs to the unit cube, so $\abs{y_i}\ll 1/\abs{\g_i}$. Letting $L_i=1/\abs{\g_i}$, we conclude by noticing that
		\[
		L_1L_2L_3= \frac{1}{\abs{\g_1}\abs{\g_2}\abs{\g_3}}\leq \frac{1}{\det(\Lambda')}=\frac{1}{\det(A)\det(\Lambda)}\ll \frac{\meas{E}}{\det(\Lambda)}. \qedhere
		\] 
	\end{proof}

Let 
\begin{equation}\label{eq:mule}	
N(q_{\c},B)=\#\left\{ \x\in \ZZp^3: q_{\c}(\x)=0, ~|\x|\leq B\right\}
\end{equation}
and let $N^0(q_{\c},B)$ denote the corresponding quantity in which any points on $\QQ$-lines contained in the conic $q_\c=0$ are deleted from the count. We now have everything in place to record the following upper bound.
	
	\begin{lemma}\label{lem:countM1}
		Let $\c=(c_1,c_2,c_3,c_4)\in\ZZ_{\prim}^4$ with $\abs{c_4}=\abs{\c}$. For all $\ve>0$, we have
		\[
		N^0(q_{\c},B)\ll_{\ve} \|Q\|^{\ve}\abs{\c}^{\ve}\left(	1+ B\min\left(1,\frac{\abs{Q^*(\c)}^{\frac 13}\|Q\|}{c_4^{\frac 23}\Delta_Q^{\frac{1}{2}}}\right)\frac{ g_{\c}^{\frac{1}{2}}T^{\frac{1}{2}} d_\c'^{\frac{1}{2}}}{\abs{c_4}^{\frac{1}{3}}\abs{Q^*(\c)}^{\frac{1}{3}}}\right).
		\]
	\end{lemma}
	
	\begin{proof}
		If $q_\c$ is singular, then it splits into two lines. If the lines are defined over $\QQ$, then $N^0(q_\c,B)=0$. If the lines are not rational, then each of them contains at most one rational point, so that $N^0(q_\c,B)=O(1)$ in this case.
		
		Assume that $q_\c$ is nonsingular. By Lemmas~\ref{lem:ellips} and~\ref{lem:lattices}, points counted by $N^0(q_{\c},B)$ lie on $O_{\ve}((\abs{\c}\|Q\|)^{\ve})$ intersections of a lattice with an ellipsoid centred at the origin. Fix one such intersection and let $q_\c'(\y)=q_\c(y_1\f_1+y_2\f_2+y_3\f_3)$, with $\f_1,\f_2,\f_3$ as in Lemma~\ref{lem:latell}. Since $\f_1,\f_2,\f_3$ are integer vectors, it follows that $(x_1,x_2,x_3)$
is nonprimitive if $(y_1,y_2,y_3)$ is nonprimitive. Hence it follows from Lemma~\ref{lem:latell} 
and~\cite[Lemma 2.5]{BrowningHeathBrown} that 
	\[
		N^0(q_{\c},B)\leq \#\{\y\in\ZZ_{\prim}^3: q_\c'(\y)=0, \abs{y_i}\leq L_i \}\ll 1+(L_1L_2L_3)^{\frac 13}.
		\] 
		We conclude by combining Lemmas~\ref{lem:ellips},~\ref{lem:lattices}, and ~\ref{lem:latell}.
	\end{proof}
	
	It remains to sum this bound over all $\c\in M_1(C,g,R,h)$,
	in the notation of \eqref{eq:defM1}.
	 Recalling that we defined $m_Q=\min_{1\leq i\leq 4}\{\abs{A_i}\}$ at the start of Section~\ref{sec:quad}, we proceed by proving the following result.
	
	\begin{lemma}\label{lem:contM1}
		Let $C,g,R,h\in \ZZ_{\geq 1}$ and let $\ve>0$. Then 
		\begin{align*}
			&\sum_{\c\in M_1(C,g,R,h)}N^0(q_{\c},B)\ll_{\ve}
			(BC\|Q\|)^{\ve}\left(\#M_1(C,g,R,h) +
			BT^{\frac{1}{2}}E_1(C,Q)
\right),
		\end{align*}
where
\[
E_1(C,Q)=			\frac{C^3\|Q\|^{\frac{1}{2}}}{m_Q^{\frac{1}{6}}\Delta_Q^{\frac{1}{3}}}+\frac{C\min\left\{\|Q\|^{\frac 14},C^{\frac 12}\right\}\|Q\|}{\Delta_Q^{\frac{1}{2}}}.
\]
	\end{lemma}
	
	\begin{proof}
For any 	$\c \in M_1(C,g,R,h)$, Lemma~\ref{lem:countM1} implies that 
		\[
		N^0(q_{\c},B)\ll_{\ve} (C\|Q\|)^{\ve}\left(1+	B\min\left(1,\frac{R^{\frac 13}\|Q\|}{C^{\frac 23}\Delta_Q^{\frac{1}{2}}}\right)\frac{ g^{\frac{1}{2}}T^{\frac{1}{2}} h^{\frac{1}{2}}}{C^{\frac{1}{3}}R^{\frac{1}{3}}}\right).
		\]
		In order for $M_1(C,g,R,h)$ to be nonempty, we need to have 
		\[
		g^{\square}\leq g\ll \min\left\{\abs{A_4},C^2\right\}\leq \min\left\{\|Q\|,C^2\right\}
		\]
and 
		\[
		R\ll \frac{\Delta_QC^2}{m_Q}.
		\]
Thus, in the light of Lemma~\ref{lem:cardM}, we see that
		\begin{align*}
		\sum_{\c\in M_1(C,g,R,h)}N^0(q_{\c},B)
		\ll~& (C\|Q\|)^{\ve}\frac{\Upsilon}{g^{\frac{1}{2}}h^{\frac{1}{2}}}\cdot
\left(	B\min\left(1,\frac{R^{\frac 13}\|Q\|}{C^{\frac 23}\Delta_Q^{\frac{1}{2}}}\right)\frac{ g^{\frac{1}{2}}T^{\frac{1}{2}} h^{\frac{1}{2}}}{C^{\frac{1}{3}}R^{\frac{1}{3}}}\right)\\
&		+(C\|Q\|)^{\ve}\#M_1(C,g,R,h).
\end{align*}
where 
\[
\Upsilon=
\frac{C^3R^{\frac{1}{2}} \|Q\|^{\frac{1}{2}}}{\Delta_Q^{\frac{1}{2}}}+
C^2\min\left\{\|Q\|^{\frac 14},C^{\frac 12}\right\}.
\]
Thus 
		\begin{align*}
		\sum_{\c\in M_1(C,g,R,h)}N^0(q_{\c},B)
		\ll~& (C\|Q\|)^{\ve} BT^{\frac{1}{2}}\Upsilon 
\left(	\min\left(1,\frac{R^{\frac 13}\|Q\|}{C^{\frac 23}\Delta_Q^{\frac{1}{2}}}\right)\frac{ 
1}{C^{\frac{1}{3}}R^{\frac{1}{3}}}\right)\\
&		+(C\|Q\|)^{\ve}\#M_1(C,g,R,h).
\end{align*}
		We clearly have
		\[
		\frac{R^{\frac 13}\|Q\|}{C^{\frac 23}\Delta_Q^{\frac{1}{2}}C^{\frac{1}{3}}R^{\frac{1}{3}}}= \frac{\|Q\|}{C\Delta_Q^{\frac{1}{2}}},
		\]
whence
		\[
		C^2\min\left\{\|Q\|^{\frac 14},C^{\frac 12}\right\}
		\left(	\min\left(1,\frac{R^{\frac 13}\|Q\|}{C^{\frac 23}\Delta_Q^{\frac{1}{2}}}\right)
		\frac{1}{C^{\frac{1}{3}}R^{\frac{1}{3}}}\right)\leq\frac{C\min\left\{\|Q\|^{\frac 14},C^{\frac 12}\right\}\|Q\|}{\Delta_Q^{\frac{1}{2}}}.
		\]
		Moreover, on recalling that $R\ll C^2\Delta_Q/m_Q$, we have
		\[
		\left(\frac{C^3R^{\frac{1}{2}} \|Q\|^{\frac{1}{2}}}{\Delta_Q^{\frac{1}{2}}C^{\frac{1}{3}}R^{\frac{1}{3}}}\right)
		\min\left(1,\frac{R^{\frac 13}\|Q\|}{C^{\frac 23}\Delta_Q^{\frac{1}{2}}}\right)
		\ll \left(\frac{C^3R^{\frac{1}{2}} \|Q\|^{\frac{1}{2}}}{ \Delta_Q^{\frac{1}{2}}}\right)\frac{1}{C^{\frac{1}{3}}R^{\frac{1}{3}}}\ll \frac{C^3\|Q\|^{\frac{1}{2}}}{m_Q^{\frac{1}{6}}\Delta_Q^{\frac{1}{3}}},
		\]
which completes the proof. 
	\end{proof}
	
	\subsection{One vanishing coordinate}\label{sec:one}

	In the previous section, we focused on the contribution from vectors $\c\in \ZZp^4$ in which $c_1\cdots c_4\neq 0$. We must now contend with the contribution from those $\c$ for which 
	$c_1\cdots c_4=0$, beginning here with the $\c$ such that $c_1=0$, $c_2c_3c_4\neq 0$, and $\abs{\c}=\abs{c_4}$. 
	Recalling 
	the definition \eqref{eq:defDc} of 
	$D_{\c}$, we now see that 
	\[
\begin{small}
	D_{\c}=\begin{pmatrix}
		c_4^2(A_2A_3c_4^2+A_2A_4c_3^2+A_3A_4c_2^2) 	& 0& 0 \\
		0& 	c_4^2(A_1A_3c_4^2+A_1A_4c_3^2) &- c_4^2c_2c_3A_1A_4 \\
		0& -c_4^2c_2c_3A_1A_4& c_4^2(A_1A_2c_4^2+A_1A_4c_2^2)\\
	\end{pmatrix}.
\end{small}
	\]
Recall that $g_\c=\gcd(A_4,c_4^2)$ and that $T_\c$ is given by \eqref{eq:Tcc}, with $A_4'=A_4/g_\c$.
Let $
	C_4=c_4^2/g_{\c}$
and 
\[
\gamma_\c=\frac{A_2A_3C_4+A_2A_4'c_3^2+A_3A_4'c_2^2}{T_{\c}}.
\]
Furthermore, let 
$
\delta_{\c}=\gcd\left(A_1,\gamma_\c\right).
$ Then we can write $c_4^2g_{\c}T_{\c}\gamma_{\c}$ for the top left entry of $D_\c$.
Recall that $d_\c$ is the greatest common divisor of the entries of $D_\c$. In this section we take 
$d_\c''$ to be the 
greatest common divisor of the entries of $D_\c'/\delta_\c$, where 
	\[
	D_\c'=\frac{D_\c}{c_4^2g_{\c}T_{\c}},
	\] 
	We now establish the following key divisibility property. 
	
	\begin{lemma}\label{lem:goat}
	We have 
		$d_\c''\mid c_2c_3$.
	\end{lemma}
	
	\begin{proof}
	The nonzero entries of $D_\c'$ are the integers
	\[
	\gamma_\c,\qquad
	\frac{A_1(A_3C_4+A_4'c_3^2)}{T_{\c}},\qquad
	-\frac{A_1A_4'c_2c_3}{T_{\c}},\qquad
	\frac{A_1(A_2C_4+A_4'c_2^2)}{T_{\c}}.
	\]

	Fix a prime $p$, and write
	\[
	\alpha_i=v_p(A_i)\quad (1\leq i\leq 3),\qquad
	\alpha_4=v_p(A_4'),\qquad
	\tau=v_p(T_{\c}),\qquad
	\sigma=v_p(\gamma_\c).
	\]
	Since
$
	\gcd(A_1,A_2,A_3,A_4')=1$,
	at least one of the four numbers
	$\alpha_1,\alpha_2,\alpha_3,\alpha_4$ is zero. Moreover, we clearly have 
	\[
	\tau=\min_{1\leq i<j\leq 4}(\alpha_i+\alpha_j), \qquad 
	v_p(\delta_{\c})=\min\{\alpha_1,\sigma\}.
	\]

	Let
	$
	m_p=v_p(d_{\c})-v_p(c_4^2g_{\c}T_{\c})$ be the $p$-adic valuation of the gcd of the
	four entries displayed above.
	If $\sigma\leq\alpha_1$, then the first entry gives
	$
	m_p\leq \sigma=v_p(\delta_{\c})$.
If $\sigma>\alpha_1$, then
$	v_p(\delta_{\c})=\alpha_1$.
	If $\alpha_4\leq\tau$, the off-diagonal entry gives
	\[
	\begin{aligned}
	m_p
	\leq
	v_p\left(\frac{A_1A_4'c_2c_3}{T_{\c}}\right)
	&=\alpha_1+\alpha_4+v_p(c_2c_3)-\tau\\
	&\leq\alpha_1+v_p(c_2c_3).
	\end{aligned}
	\]
If $\alpha_4>\tau$, then $v_p(C_4)=0$, since $A_4'$ is coprime to $c_4^2/g_{\c}$.
Furthermore, there is an index $j\in\{2,3\}$ such that 
$	\alpha_j\leq\tau$.
	For this $j$, the two summands in $A_jC_4+A_4'c_j^2$
	have valuations
	$\alpha_j$ and $\alpha_4+2v_p(c_j)>\tau\geq\alpha_j$,
	respectively. Hence
	$
	v_p(A_jC_4+A_4'c_j^2)=\alpha_j$ and so the corresponding diagonal entry has valuation
	\[
	\alpha_1+\alpha_j-\tau\leq\alpha_1.
	\]
	Thus $m_p\leq\alpha_1=v_p(\delta_{\c})$ in this case.

	We have proved in every case that
	$
	m_p\leq v_p(\delta_{\c})+v_p(c_2c_3)$.
	Using the definition of $d_{\c}''$, it follows that
	\[
	\begin{aligned}
	v_p(d_{\c}'')
	&=\max\{0,m_p-v_p(\delta_{\c})\}\leq v_p(c_2c_3),
	\end{aligned}
	\]
	as required.
\end{proof}

For $C,g,R,h\in \ZZ_{\geq 1}$, we put
	\begin{align*}
		M_0(C,g,R,h)=
		\left\{\c\in \ZZ_{\prim}^4:
		\begin{array}{l}
		c_1=0,~\abs{\c}=\abs{c_4}\sim C,~g_{\c}=g, ~d_\c''=h, \\
				|Q^{*}(\c)|\sim R, ~c_2c_3\neq 0
				\end{array}\right\}.
	\end{align*}
The following result upper bounds this quantity.

	\begin{lemma}\label{lem:cardM0}
		Let $C,g,R,h\in \ZZ_{\geq 1}$. For all $\ve>0$, we have
		\[
		\#	M_0(C,g,R,h)\ll_{\ve} (RC)^{\ve}\left(\frac{C^2R^{\frac{1}{2}} \|Q\|^{\frac{1}{2}}}{g^{\frac{1}{2}} \Delta_Q^{\frac{1}{2}}h}+\frac{C^{\frac 32}(g^{\square})^{\frac{1}{4}}}{g^{\frac{1}{2}}h^{\frac{1}{2}}}\right).
		\]
	\end{lemma}
	
	\begin{proof}
We follow the proof of Lemma~\ref{lem:cardM}, invoking 
Lemma~\ref{lem:goat} to deduce that 
		\[
		\#M_0(C,g,R,h)\leq \sum_{h_2h_3=h} \#M_0(C,g,R,h_2,h_3),
		\]
where 
		$M_0(C,g,R,h_2,h_3)$ is the set of $\c\in M_0(C,g,R,h)$ such that $h_i\mid c_i$, for $i=2,3$.
				Set $h_4=g/\sqrt{g^{\square}}$. 
		It follows from Lemma~\ref{lem:cimidhi} that $h_4\mid c_4$. Hence we can define integers $c_k'=c_k/h_k$, for $2\leq k\leq 4$.
		Putting $B_i=A_2A_3A_4h_i^2/A_i$, we see that 
		\[
		Q^*(\c)=A_1(B_2c_2'^2+B_3c_3'^2+B_4c_4'^2),
		\]
whence
		\[
		\#M_0(C,g,R,h_2,h_3)\leq \#\left\{{\c'}=(c_2',c_3',c_4')\in(\ZZ_{\neq 0})^3:\abs{\sum_{i=2}^4 B_ic_i'^2}\leq \frac{2R}{|A_1|}, |c_i'|\leq \frac{2C}{h_i}\right\}.
		\]
		Let $j$ be such that $\abs{B_j}=\min_i\{\abs{B_i}\}$. Then it follows from 
Corollary~\ref{lem:NRX} that 
		\begin{align*}
			\#M_0(C,g,R,h_2,h_3)
&\ll_\ve \frac{C^{\frac{3}{2}+\ve}R^{\ve}(g^\square)^{\frac{1}{4}}}{h^{\frac{1}{2}}g^{\frac{1}{2}}}+
\frac{C^{3}(g^\square)^{\frac{1}{2}}}{h g} \left(\frac{Rh_j^2}{|A_1B_j|C^2}\right)^{\frac 12}\\
			&\ll_\ve 
			\frac{C^{\frac{3}{2}+\ve}R^{\ve}(g^{\square})^{\frac{1}{4}}}{h^{\frac 12}g^{\frac 12}}
			+\frac{C^2}{hg^{\frac 12}} \left(\frac{R\|Q\|}{\Delta_Q}\right)^{\frac 12},
		\end{align*}
		which is enough to complete the proof of the lemma.
	\end{proof}

Recall the definition 
\eqref{eq:mule}	 of $N(q_{\c},B)$ and the allied quantity $N^0(q_{\c},B)$, in which points on $\QQ$-lines are deleted. We may now record the following result, which 
is deduced from Lemma~\ref{lem:cardM0} in exactly the same way that 
Lemma~\ref{lem:contM1} is deduced from 
Lemma~\ref{lem:countM1}.
	
	\begin{lemma}\label{lem:contM0}
		Let $C,g,R,h\in \ZZ_{\geq 1}$ and let $\ve>0$. Then 
		\begin{align*}
			\sum_{\c\in M_0(C,g,R,h)}N^0(q_{\c},B)\ll_{\ve}
			(C\|Q\|)^{\ve}\left(\#M_0(C,g,R,h)+BT^{\frac{1}{2}}
			E_0(C,Q)\right)
\end{align*}
where 
\[
E_0(C,Q)=
			\frac{C^2\|Q\|}{m_Q^{\frac{1}{6}}\Delta_Q^{\frac{1}{3}}}+\frac{C^{\frac 12}\min\left\{\|Q\|^{\frac{1}{4}}, C^{\frac 12}
				\right\}\|Q\|^{\frac{3}{2}}}{\Delta_Q^{\frac{1}{2}}}.
\]
	\end{lemma}

	\subsection{At least two vanishing coordinates}
	
		We focus now on the contribution from vectors $\c=(c_1,c_2,c_3,c_4)$ with at least two coefficients equal to zero, letting 
	\[
	M'(C)=\{\c\in\ZZ_{\prim}^4: \abs{\c}=|c_4|\leq C, ~c_1=c_2=0\},
	\] 
for any $C\geq 1$. Note that $\#M'(C)$ has exact order $C^2$.
	
	\begin{lemma}\label{lem:contM'}
Let $C\in \ZZ_{\geq 1}$ and let $\ve>0$. Then 
		\[
		\sum_{\c\in M'(C)}N^0(q_{\c},B)\ll_{\ve} 
		\left(C^2+BC\right)B^\ve.
		\]
	\end{lemma}
	
	\begin{proof}
		Assume that $c_3=0$. Then $\c=(0,0,0,\pm 1)$ and 
				$\c\cdot \x=0$ implies $x_4 = 0$. Therefore
		\[
		N^0(q_\c,B) = \#\{\x \in \ZZ^3_{\prim}: A_1x_1^2 + A_2x_2^2 + A_3x_3^2 = 0, |\x|\leq B\}=O(B),
		\]
by standard estimates for counting rational points of bounded height on conics.
This case therefore makes the overall contribution $O(B)$.
		Assume $c_3\neq 0$. Then $\c\cdot \x=0$ implies $c_4\mid x_3$. Therefore, there are $O(B/|c_4|+1)$ choices for $x_3$, after which $x_4$ is uniquely determined. Now for fixed $x_3, x_4$, the number of $x_1, x_2$ such that $Q(\x) = \x \cdot \c = 0$ 
		that can possibly contribute to $N^0(q_{\c},B)$ 
		is $O_{\ve}(B^{\eps})$. Hence \[N^0(q_{\c},B)\ll_{\ve} \frac{B^{1+\eps}}{|c_4|}+B^{\eps}.
		\] 
		Taking a sum over $|c_3|,|c_4| \leq C$, we obtain a contribution of $O_{\ve}(
		(C^2+BC)B^\ve)$.
	\end{proof}

\subsection{Combining all the estimates}

	Let $C,g,R,h\in \ZZ_{\geq 1}$. Recall that $g_\c$ was defined at the beginning of Section~\ref{sec:quad}, that $d_\c'$ was defined in Section~\ref{sec:no}, and that 
	$d_\c''$ was defined in Section~\ref{sec:one}. 
We take 
\begin{equation}\label{eq:cow}
\tilde d_\c=\begin{cases}
d_\c' & \text{ if $c_1c_2c_3c_4\neq 0$,}\\
d_\c'' & \text{ if $c_1=0$ and $c_2c_3c_4\neq 0$,}\\
1 & \text{ otherwise.}
\end{cases}
\end{equation}
Let
	\begin{equation}\label{eq:sheep}
		M(C,g,R,h)=\left\{\c\in \ZZ_{\prim}^4:\abs{\c}\sim C,~g_{\c}=g,~\tilde d_\c=h,~|Q^*(\c)|\sim R\right\}.
	\end{equation}
Then we are now ready to record the final upper bound for the average of 
$N^0(q_{\c},B)$ over $\c\in M(C,g,R,h)$, where 
$N^0(q_{\c},B)$ is defined immediately after \eqref{eq:mule}.
	
	\begin{lemma}\label{lem:finalcontr}
		Let $C,g,R,h\in \ZZ_{\geq 1}$ and let $\ve>0$. Then 
		\begin{align*}
			&\sum_{\c\in M(C,g,R,h)}N^0(q_{\c},B)\ll_{\ve}
 (BC\|Q\|)^{\ve}\left(\#M(C,g,R,h) + C^2+BC+
 BT^{\frac{1}{2}}E(C,Q) \right),
 \end{align*}
where
\[
E(C,Q)=\frac{C^2\|Q\|}{m_Q^{\frac{1}{6}}\Delta_Q^{\frac{1}{3}}}+\frac{C^{\frac 12}\min\left\{\|Q\|^{\frac{1}{4}}, C^{\frac 12}
				\right\}\|Q\|^{\frac{3}{2}}}{\Delta_Q^{\frac{1}{2}}}+
\frac{C^3\|Q\|^{\frac{1}{2}}}{m_Q^{\frac{1}{6}}\Delta_Q^{\frac{1}{3}}}.
\]
\end{lemma}
	
	\begin{proof}
		Up to rearranging the indices, we can assume $\abs{c_1}\leq \abs{c_2}\leq \abs{c_3}\leq \abs{c_4}$. If $c_1\neq 0$, we use Lemma~\ref{lem:contM1}. If $c_1=0$ and $c_2\neq 0$, we use Lemma~\ref{lem:contM0}. If $c_1=c_2=0$, we use Lemma~\ref{lem:contM'}. 
Recalling that $\#M'(C)$ has exact order $C^2$, 
this leads to the bound
\begin{align*}
\sum_{\c\in M(C,g,R,h)}N^0(q_{\c},B)&\ll_{\ve}
 (BC\|Q\|)^{\ve}\left(\#M(C,g,R,h) + C^2+BC+
 BT^{\frac{1}{2}}\tilde E(C,Q) \right),
 \end{align*}
 where 
 \begin{align*}
\tilde E(C,Q)
&=E_0(C,Q)+E_1(C,Q)\ll E(C,Q),
\end{align*}
since 
				\[
		\frac{C\min\left\{\|Q\|^{\frac 14},C^{\frac 12}\right\}\|Q\|}
		{\Delta_Q^{\frac{1}{2}}}
		\leq \frac{C^2\|Q\|}{m_Q^{\frac{1}{6}}\Delta_Q^{\frac{1}{3}}}. \qedhere
		\]
	\end{proof}
	\section{Uniform bounds for quadrics: conclusion}\label{sec:quad''}
	In this section we shall deduce Theorem~\ref{thm:main}. Recall the definition \eqref{eq:ham} of 
	$N(Q,B)$ and $N^0(Q,B)$. 
	If $\|Q\|\ll B^{20}$ does not hold, then~\cite[Lemma 5]{BHBS-duke} yields $N(Q,B)\ll B$, which is satisfactory. Thus, from now on, we can assume $\|Q\|\ll B^{20}$.
Let $\c\in \ZZp^4$ and recall the definition of the ternary quadratic form $q_\c$ from the start of Section~\ref{sec:quad}. For given $d\in \ZZ_{\geq 1}$, we define 
\[
	N_d^0(q_\c,B)=\#\left\{\x\in \ZZ^3:\gcd(x_1,x_2,x_3)=d, ~q_\c(\x)=0,~\abs{\x}\leq B\right\},
\]
		if $q_\c$ is irreducible over $\QQ$, with $N_d^0(q_\c,B)=0$ if $q_\c$ factors over $\QQ$.
	
	\begin{lemma}\label{lem:finalcontrd}
		Let $C,g,R,h\in \ZZ_{\geq 1}$ and let $\ve>0$. Assume that $\|Q\|\ll B^{20}$ and $C\ll B^{\frac 13}$. Then 
\begin{align*}
\sum_{d\mid A_4}
\sum_{\c\in M(C,g,R,h)}N_d^0(q_{\c},B)\ll_{\ve}
 B^{\ve}\left(B^{\frac{4}{3}} +
T^{\frac{1}{2}}F(B,Q) \right),
 \end{align*}
 		where $M(C,g,R,h)$ is given by \eqref{eq:sheep} and 
 \begin{equation}\label{eq:bull}
 F(B,Q)=
 \frac{B^{\frac 76}\min\left\{\|Q\|^{\frac{1}{4}}, B^{\frac 16}
				\right\}\|Q\|^{\frac{3}{2}}}{\Delta_Q^{\frac{1}{2}}}+
 \frac{B^{\frac{5}{3}}\|Q\|}{m_Q^{\frac{1}{6}}\Delta_Q^{\frac{1}{3}}}+
 \frac{B^2\|Q\|^{\frac{1}{2}}}{m_Q^{\frac{1}{6}}\Delta_Q^{\frac{1}{3}}}.
\end{equation}
	\end{lemma}
	
	\begin{proof}
		First of all, notice that $N_d^0(q_\c,B)=N^0(q_\c,B/d)$. Thus it follows from Lemma~\ref{lem:finalcontr} that 
		\begin{align*}
			\sum_{\c\in M(C,g,R,h)}N_d^0(q_{\c},B)\ll_{\ve}
 B^{\ve}\left(\#M(C,g,R,h) + C^2+ \frac{BC}{d}+
 \frac{BT^{\frac{1}{2}}E(C,Q)}{d} \right),
 \end{align*}
where $BC\ll B^{\frac 43}$ and 
\[
E(C,Q)\ll \frac{B^{\frac{2}{3}}\|Q\|}{m_Q^{\frac{1}{6}}\Delta_Q^{\frac{1}{3}}}+\frac{B^{\frac 16}\min\left\{\|Q\|^{\frac{1}{4}}, B^{\frac 16}
				\right\}\|Q\|^{\frac{3}{2}}}{\Delta_Q^{\frac{1}{2}}}+
\frac{B\|Q\|^{\frac{1}{2}}}{m_Q^{\frac{1}{6}}\Delta_Q^{\frac{1}{3}}}.
\]
Recalling \eqref{eq:sheep}, it is clear that $\#M(C,g,R,h)\ll B^{\frac{4}{3}}$ if $C\ll B^{\frac 13}.$
We can bound the total number $d\mid A_4$ by $O_\ve(B^\ve)$ by the divisor bound. 
The statement of the lemma easily follows.
	\end{proof}
	
We are now ready to prove the following result, 
which is a more technical version of Theorem~\ref{thm:main} that is sharper in certain ranges.

	\begin{theorem}\label{prop:main}
Let $\ve>0$. Then 
		\[
		N^{0}(Q,B)\ll_{\ve} B^{\ve}\left(
		B^{\frac43}
		+
		\frac{B^{\frac76}T^{\frac12}
			\min\{\|Q\|^{\frac14},B^{\frac16}\}\|Q\|^{\frac32}}{\Delta_Q^{\frac12}}
		+
		\frac{B^{\frac53}T^{\frac12}\|Q\|}{m_Q^{\frac16}\Delta_Q^{\frac13}} 
		+
		\frac{B^{2}T^{\frac12}\|Q\|^{\frac12}}{m_Q^{\frac16}\Delta_Q^{\frac13}}\right).
		\]
	\end{theorem}
	\begin{proof}
For any $\c\in \ZZ_{\prim}^4$, let 
	$N^0(Q,B,\c)$ be the number of points counted by $N^0(Q,B)$ which satisfy the additional constraint $\c\cdot \x=0$.
	As explained in~\cite[Lemma 2.1]{BrowningHeathBrown}, it follows from Siegel's lemma that any solution to $Q(\x) = 0$ with $\x \in \ZZ^4$ and $|\x| \leq B$ must satisfy 
the linear equation	$\c\cdot \x=0$ for some $\c \in \ZZ^4_{\prim}$ with $|\c| \ll B^{\frac 13}$. Thus
\[
		N^0(Q,B)\leq \sum_{\substack{\c \in \ZZ^4_{\prim}\\ |\c|\ll B^{\frac 13}}}N^0(Q,B,\c).
\]
Suppose without loss of generality that $\c \in \ZZ^4_{\prim}$ is such that 
$\abs{\c}=\abs{c_4}$. 
		Let $\x\in \ZZ_{\prim}^4$ be such that $Q(\x)=\x\cdot \c=0$. Then $q_\c(x_1,x_2,x_3)=0$, in the notation of Section~\ref{sec:quad}. 
			 If $q_\c$ factorises over $\QQ$ then $\x$ is not counted by 
			 $N^0(Q,B,\c)$. If not, then $\x$ is counted by $N_d^0(q_{\c},B)$, for some $d\in \ZZ_{\geq 1}$. Moreover, since $\x$ is primitive and $Q(\x)=0$, we must have $\gcd(x_1,x_2,x_3)\mid A_4$. Hence
\[
		N^0(Q,B,\c)\leq \sum_{d\mid A_4} N_d^0(q_{\c},B).
\]
	
		Notice that $\gcd(A_4,c_4^2)\mid \Delta_Q$ and $d_{\c}\mid \Delta_Q$, where the latter follows from the argument at the top of page 11 in~\cite{BrowningHeathBrown}. In particular, we must have $\tilde d_\c\mid \Delta_Q$, in the notation of \eqref{eq:cow}. 
		Moreover, $Q^*(\c)\ll \Delta_Q \abs{\c}^2/m_Q$. 
			Fixing $g,h\mid \Delta_Q$, $1\leq C\ll B^{\frac 13}$, and $1\leq R\ll \Delta_Q C^{2}/m_Q$, it now follows from Lemma~\ref{lem:finalcontrd} that
		\begin{align*}
			\sum_{\c\in M(C,g,R,h)}N^0(Q,B,\c)\ll_{\ve} 
			B^{\ve}\left(
			B^{\frac43}
			+
T^{\frac12} F(B,Q)\right),
\end{align*}
where $F(B,Q)$ is given by \eqref{eq:bull}.
Notice that this inequality holds even when $R=0$. Indeed, if $Q^*(\c)=0$, then $q_{\c}(x_1,x_2,x_3)=0$ is singular and so splits as union of two lines, so that $N^{0}(q_\c,B)=O(1)$. 
Hence it follows that 
\[
\sum_{\c\in M(C,g,0,h)}N^0(Q,B,\c)\ll
		B^{\frac43}.
		\] 
		We conclude by summing over all possible choices of $g$ and $h$, and allowing $R$ and $C$ to run over all powers of $2$, for which 
		$1\leq C\ll B^{\frac 13}$ and $R\leq \Delta_Q C^{2}/m_Q$.
	\end{proof}
	
	\begin{proof}[Proof of Theorem~\ref{thm:main}]
		Assume first that $\|Q\|^3\geq \Delta_QB^2$ and observe that 
		\eqref{eq:comtat} implies the bound 
		\[
		N^0(Q,B)\ll B^{2}\leq \frac{B^{\frac{4}{3}}\|Q\|}{\Delta_Q^{\frac{1}{3}}}\leq \frac{B^{\frac 53}T^{\frac{1}{2}}\|Q\|}{\Delta_Q^{\frac{1}{3}}}.
		\]
		Alternatively, if $\|Q\|^3\leq \Delta_QB^2$, we apply Theorem~\ref{prop:main}. But then we just need to observe that 
		\[
		\frac{B^{\frac{7}{6}}\min\{B^{\frac{1}{6}},\|Q\|^{\frac{1}{4}}\}T^{\frac{1}{2}}\|Q\|^{\frac{3}{2}}}{\Delta_Q^{\frac{1}{2}}}\leq\frac{B^{\frac{4}{3}}T^{\frac{1}{2}}\|Q\|^{\frac{3}{2}}}{\Delta_Q^{\frac{1}{2}}}\leq \frac{B^{\frac 53}T^{\frac{1}{2}}\|Q\|}{\Delta_Q^{\frac{1}{3}}}. \qedhere
		\]
	\end{proof}
	
	\section{Application to some biprojective fivefolds} \label{final}
	
\subsection{Estimates for points on cubic surfaces}
	Given a surface $S\subset \PP^3$, we denote by $N(S,B)$ the cardinality of the set of rational points on $S$ of exponential height bounded by $B$, 
	and we let $N^{0}(S,B)$ denote the cardinality of the subset of points that do not lie on any rational lines contained in $S$. 
	The following coefficient-sensitive estimate follows by combining work of Walsh
	on the 
	 global determinant method \cite{walsh} with a refined 
determinant method estimate of	 Salberger~\cite{salberger}. 
Since the precise dependence on the coefficient height is important here, we include the relevant modification.
	
	\begin{lemma}\label{lem:cubicab}
		Let $B_1, B_2, B_3, B_4\in \ZZ_{\neq 0}$,
				with $\gcd(B_1,B_2,B_3,B_4)=1$,
				 and let $S\subset\mathbb{P}^3$ be the smooth cubic surface 
		$
		B_1x_1^3+B_2x_2^3+B_3x_3^3+B_4x_4^3=0.
		$
Then 		\[
		N^{0}(S,B)\ll_\ve B^{\ve}\left(\frac{B^{\frac {12}{7}}}{M^{\frac{4}{21}}}+\frac{B^{\frac{15}{28}+1}}{M^{\frac{5}{84}}}+B \right),
		\]
		where $M=\max_{1\leq i\leq 4}\{\abs{B_i}\}$.
	\end{lemma}
	
	\begin{proof}
		By~\cite[Theorem 1.3]{walsh}, points of bounded height lie on a curve of degree 
		\[
		D\ll_\ve\left(\frac{B}{M^{\frac 19}}\right)^{\frac{\sqrt{3}}{2}}+B^{\ve}.
		\] The exponent $\sqrt{3}/2$ can be improved to $6/7$ using work of Salberger~\cite{salcubic}, as we briefly explain now. By encoding~\cite[Lemma 16.11]{salcubic} in the proof of~\cite[Theorem 2.2]{walsh} (as in the argument at the beginning of~\cite[Section 3]{bv-quad}, for example), one can substitute $(n!)^{\frac{1}{n}}n/(n+1)=2\sqrt{2}/3$ with $7\sqrt{2}\sqrt{3}/18$. Proving~\cite[Theorem 1.3]{walsh} with this new bound, one gets
		\[
		D\ll_\ve\left(\frac{B}{M^{\frac 19}}\right)^{\frac{6}{7}}+B^{\ve}.
		\] 
		Points on conics on the curve are $O_{\ve}(B^{1+\ve}D^{\frac 58})$ by~\cite[Lemma 5.3]{salberger}. Points on irreducible curves of degree at least $3$ are $O_{\ve}(B^{\ve}(D^2+DB^{\frac 23}))$, using~\cite{curves} as in~\cite[Lemma 2.1]{smooth}. We conclude by noticing that
		\[
		DB^{\frac 23}\leq D^{\frac{13}{12}}B^{\frac 23}=\left(D^2\right)^{\frac 13}\left(BD^{\frac 58}\right)^{\frac 23}\leq D^2+BD^{\frac 58}. \qedhere
		\]
	\end{proof}

\subsection{The main argument}

We are now ready to start tackling the proof of Theorem~\ref{thm:appl}, which concerns the biprojective fivefold \eqref{eq:eq}, for $\c\in \ZZp^4$ with $c_1\dots c_4\neq 0$.
	Recall the definition of the thin set $W_1$ from \eqref{eq:thin1}. 
	If $a=b=3$, let
\[
W_3=
\bigcup_{i\neq j}
\left\{
(x,y)\in Z(\QQ)\setminus W_1:
c_ix_i^3y_i^3+c_jx_j^3y_j^3=0
\right\}.
\]
Suppose that $(a,b)\neq(3,3)$. If $a\neq3$, let
$W_3^{(x)}=\emptyset$, while if $a=3$, let
\[
W_3^{(x)}
=
\bigcup_{i\neq j}
\left\{
(x,y)\in Z(\QQ)\setminus W_1:
\frac{c_i y_i^b}{c_jy_j^b}\in\QQ^{\times3}
\right\}.
\]
Similarly, if $b\neq3$, let $W_3^{(y)}=\emptyset$, while if
$b=3$, let
\[
W_3^{(y)}
=
\bigcup_{i\neq j}
\left\{
(x,y)\in Z(\QQ)\setminus W_1:
\frac{c_i x_i^a}{c_jx_j^a}\in\QQ^{\times3}
\right\}.
\]
In this case put $W_3=W_3^{(x)}\cup W_3^{(y)}$.
Finally, put
\[
W=W_1\cup W_3.
\]
Then we shall take $U=Z(\QQ)\setminus W$
in Theorem~\ref{thm:appl}, 
where $W$ is a thin subset of $Z(\QQ)$.

	Let $\X=(X_1,X_2,X_3,X_4)\in \RR_{>0}^4$ and $\Y=(Y_1,Y_2,Y_3,Y_4)\in \RR_{>0}^4$. We say $\x\sim \X$ if $|x_i|\sim X_i$ for $1\leq i\leq 4$.
	Define
	\begin{equation}\label{eq:fox}
	N(U,\X,\Y)=\#\left\{(\x,\y)\in \ZZ_{\prim}^4\times \ZZ_{\prim}^4: (x,y)\in U,~
	\x\sim \X,~ \y\sim \Y\right\},
	\end{equation}
	where $x=[\x]$ and $y=[\y]$ are the corresponding points in projective space and $U=Z(\QQ)\setminus W$.
We henceforth let $X=2\max_{1\leq i\leq 4}\{X_i\}$ and $Y=2\max_{1\leq i\leq 4}\{Y_i\}$.
By the height bound, we can assume 
\begin{equation}\label{eq:height}
X^{4-a}Y^{4-b}\ll B.
\end{equation}
On summing over all the dyadic intervals, we will then have 
\begin{equation}\label{eq:dog}
N(U,B)\ll (\log B)^8 \max_{
\substack{
(\X,\Y)\in \RR_{>0}^8\\ 
\text{\eqref{eq:height} holds
}}} N(U,\X,\Y)
\end{equation}
in \eqref{eq:duck}.

We proceed by collecting together various estimates
for $N(U,\X,\Y)$, for various choices of $a,b$. 
Our first bound uses basic Fourier analysis and is modelled on 
\cite[Lemma~2.8]{BHB-quadbundle}, which is used to study the case $(a,b)=(2,1)$.

	\begin{lemma}\label{lem:CSab}
For all $\ve>0$, we have 
		\[
		N(U,\X,\Y)\ll_\ve 
		\begin{cases}
		X^{2+\ve} Y^{3+\ve} & \text{ 		if $a\in \{2,3\}$,}\\ 
X^{3+\ve} Y^{3+\ve} & \text{ 		if $a=1$.}
\end{cases}
		\]
	\end{lemma}
	
	\begin{proof}
		We have
		\[
		N(U,\X,\Y)\leq \int_0^1 \sum_{\substack{0<\abs{x_i}\leq X\\0<\abs{y_i}\leq Y}} e\left(\alpha \sum_{1\leq i\leq 4} c_ix_i^ay_i^b\right)\d\alpha=\int_0^1 \prod_{i=1}^4 S_i(\alpha) \d\alpha,
		\]
		with \[S_i(\alpha)=\sum_{\substack{0<\abs{x}\leq X\\0<\abs{y}\leq Y}} e\left(\alpha c_ix^ay^b\right).\]
		H\"older's inequality yields
		\[
		N(U,\X,\Y)\leq \prod_{i=1}^4\left(\int_0^1\abs{S_i(\alpha)}^4\d\alpha\right)^{\frac 14}=\int_0^1\abs{S(\alpha)}^4\d\alpha
		\]
		with
		\[
		S(\alpha)=\sum_{\substack{0<\abs{x}\leq X\\0<\abs{y}\leq Y}} e\left(\alpha x^ay^b\right).
		\]
		
		Using Cauchy--Schwarz, we obtain
		\[
		\abs{S(\alpha)}^2\ll Y\sum_{0<\abs{y}\leq Y}\abs{\sum_{0<\abs{x}\leq X}e\left(\alpha x^ay^b\right)}^2= Y\sum_{0<\abs{y}\leq Y}\sum_{0<\abs{x_1},\abs{x_2}\leq X}e\left(\alpha y^b(x_1^a-x_2^a)\right),
		\]
whence		\begin{align*}
			N(U,\X,\Y)
			&\ll Y\sum_{0<\abs{y}\leq Y}\sum_{\substack{0<\abs{x_1},\abs{x_2},\abs{x_3},\abs{x_4}\leq X\\0<\abs{y_3},\abs{y_4}\leq Y}}\int_0^1e\left(\alpha\left(y^b(x_1^a-x_2^a)+y_3^bx_3^a-y_4^bx_4^a\right)\right) \d\alpha.
		\end{align*}
		It follows that 
		\[
			N(U,\X,\Y) \ll\ Y\#
			\left\{(\x,y,y_3,y_4)\in (\mathbb{Z}_{\ne 0})^7:
			\begin{array}{l}
									|x_i|\le X,\ |y|,|y_3|,|y_4|\le Y,\\
			y^{b}(x_1^{a}-x_2^{a}) + y_3^{b}x_3^{a} - y_4^{b}x_4^{a} = 0
			\end{array}
			\right\}.
		\]
Fix 
one of $O(X^2Y^2)$ choices for 
$x_3,x_4,y_3,y_4$. If $y_3^bx_3^a-y_4^bx_4^a\neq0$, then we have $O_\ve(X^\ve Y^\ve)$ choices for $y,x_1,x_2$ by the divisor bound in the case $a=2,3$, and $O_\ve(X^{1+\ve} Y^\ve)$ choices in the case $a=1$. If $y_3^bx_3^a-y_4^bx_4^a=0$, 
there are only $O_{\ve}(X^{1+\ve} Y^{1+\ve})$ choices for 
$x_3,x_4,y_3,y_4$ and then only $O(X Y)$ choices for $y,x_1,x_2$, since $y\neq 0$.
The statement of the lemma is now clear. 
	\end{proof}

In the next few results it will be convenient to introduce some more notation. 
	Given $\x=(x_1,x_2,x_3,x_4)$, let $S_{\x}\subset\mathbb{P}^3$ be the degree $b$ surface defined by \eqref{eq:eq}, when $\x$ is fixed. 
		Similarly, given $\y=(y_1,y_2,y_3,y_4)$, we let $S_{\y}\subset\mathbb{P}^3$ be the degree $a$ surface defined by \eqref{eq:eq} when $\y$ is fixed. 
				Recall that $U=Z(\QQ)\setminus W$, and suppose that 
$(\x,\y)$ is counted by $N(U,\X,\Y)$, so that $(x,y)\not\in W$. Since $(x,y)\notin W_1$, it follows that $S_{\x}$ and $S_{\y}$ are smooth surfaces. If $a=3$ (respectively, $b=3$), then 
it follows from a result of Segre~\cite{segre} that there are 
no rational lines on $S_{\y}$ (respectively, $S_{\x}$), since $(x,y)\not\in W_3$.

		\begin{lemma}\label{cor:a3}
Let $a=3$ and let $\ve>0$. Then 
		\[
		N(U,\X,\Y)\ll_{\ve} X^{\ve}
		\left(X^{\frac{12}{7}}Y^{4-\frac{4b}{21}}
		+X^{\frac{43}{28}}Y^{4-\frac{5b}{84}}
		+XY^4 \right).
		\]
	\end{lemma}
	
	\begin{proof}
		Fix $\y\in \ZZp^4$ such that $\y\sim \Y$. We will count points on $S_{\y}$ and then sum over all $\y\sim \Y$. 
				Since we are counting points on $U$, we can assume that $S_{\y}$ is smooth and there are no rational lines on it.
				Let 
		$g_{\y}=\gcd(c_1y_1^b,\dots, c_4y_4^b)$. 
		For all primes $p$, if $p^k$ divides $g_{\y}$, then $p^k\mid y_i^b$ for at least one $i$,
		since $\gcd(c_1,\dots,c_4)=1$. Moreover, $p^k\mid c_1\cdots c_4$, since $\gcd(y_1,\dots,y_4)=1$. In particular, it follows that 
	$g_\y\mid c_1\cdots c_4$. 
For a fixed $g\mid c_1\cdots c_4$, we 
let $\mathcal{Y}_g$ denote the set of 
$\y\in \ZZp^4$
such that $\y\sim \Y$ and 
$g_{\y}=g$, so that 
\[
\sum_{g\mid c_1\cdots c_4} \#\mathcal{Y}_g\ll Y_1Y_2Y_3Y_4. 
\]
Furthermore, we clearly have 
\[
\#\mathcal{Y}_g\ll
g^{-\frac{1}{3}}Y_1Y_2Y_3 Y_4,
\]
since $b\leq 3$. 

For any $\y\in \mathcal{Y}_g$, we put  $c_i'=c_iy_i^b/g$, for $1\leq i\leq 4$, so that $\c'\in \ZZp^4.$
It now follows from Lemma~\ref{lem:cubicab} that 
		\begin{align*}
			N(U,\X,\Y)&\ll_\ve X^\ve\sum_{
			\substack{
			g\mid c_1\cdots c_4}}\sum_{\substack{\y\in \mathcal{Y}_g}}
			\left(\frac{X^{\frac{12}{7}}}{|\c'|^{\frac{4}{21}}}
			+\frac{X^{\frac{15}{28}+1}}{|\c'|^{\frac{5}{84}}}
			+X \right).
			\end{align*}
Put 
$
C_0=|c_1\cdots c_4|$ and $\Delta_\Y=Y_1\cdots Y_4$,
and note that 
\[
|\c'|^4\geq \prod_{i=1}^4|c_i'| \gg \frac{C_0\Delta_\Y^b}{g^4}.
\]
But then we may take $
|\c'|\gg C_0^{\frac{1}{4}} \Delta_\Y^{\frac{b}{4}}/g$, obtaining 
		\begin{align*}
			N(U,\X,\Y)&\ll_\ve X^\ve
			\sum_{\substack{g\mid C_0}}
			\sum_{\substack{\y\in \mathcal{Y}_g}}
			\left(\frac{X^{\frac{12}{7}}g^{\frac{4}{21}}}{C_0^{\frac{1}{21}}
		\Delta_\Y^{\frac{b}{21}}}
			+\frac{X^{\frac{15}{28}+1}g^{\frac{5}{84}}}{C_0^{\frac{5}{336}}\Delta_\Y^{\frac{5b}{336}}}\right) 
			+X^{1+\ve}\Delta_\Y\\
			&\ll_{\ve} X^\ve\sum_{\substack{g\mid C_0}}
			\left(\frac{X^{\frac{12}{7}}g^{\frac{4}{21}-\frac{1}{3}}\Delta_\Y^{1-\frac{b}{21}}}{C_0^{\frac{1}{21}}}
			+\frac{X^{\frac{15}{28}+1}g^{\frac{5}{84}-\frac{1}{3}}\Delta_\Y^{1-\frac{5b}{336}}}{C_0^{\frac{5}{336}}}\right)
			+X^{1+\ve} \Delta_\Y\\
&\ll_\ve X^\ve		 
			\left(X^{\frac{12}{7}}Y^{4-\frac{4b}{21}} 
			+X^{\frac{15}{28}+1}Y^{4-\frac{5b}{84}}
			+XY^4 \right),
		\end{align*}
since 
\[
C_0^{-\frac{1}{21}}
\sum_{g\mid C_0}g^{\frac{4}{21}-\frac{1}{3}}\ll
C_0^{-\frac{1}{21}}\ll 1
\quad \text{ and }\quad
C_0^{-\frac{5}{336}}
\sum_{g\mid C_0}g^{\frac{5}{84}-\frac{1}{3}}\ll
C_0^{-\frac{5}{336}}\ll 1. \qedhere
\]
	\end{proof}

\begin{proof}[Proof of Theorem~\ref{thm:appl}]
We will need different arguments for each choice of $a,b\in \{1,2,3\}$, in which we identify the case 
$(a,b)$ with $(b,a)$, by symmetry.

\medskip
\noindent
\underline{The case $(a,b)=(1,1)$.}
It follows from Lemma~\ref{lem:CSab} that $N(U,\X,\Y)\ll_{\ve}X^{3+\ve}Y^{3+\ve}$. On appealing to 
\eqref{eq:height} and \eqref{eq:dog}, we obtain the desired bound
$N(U,B)\ll_{\ve} B^{1+\ve}$. 

\medskip
\noindent
\underline{The case $(a,b)=(2,1)$.}
It follows from Lemma~\ref{lem:CSab} that $N(U,\X,\Y)\ll_{\ve}X^{2+\ve}Y^{3+\ve}$, which leads to 
the desired bound $N(U,B)\ll_{\ve} B^{1+\ve}$, via 
\eqref{eq:dog}.

\medskip
\noindent
\underline{The case $(a,b)=(2,2)$.}
For a point $(x,y)\in Z(\QQ)$ 
outside
$W_1$, put $z_i=x_i y_i$ for $1\leq i\leq 4$. 
For each $\z\in \ZZ_{\neq 0}^4$, the number of factorisations $z_i=x_i y_i$
is $O_\ve (B^\ve)$, by the divisor bound. Moreover, it follows from the 
constraint $H(x,y)\leq B$ that 
$\max_{1\leq i\leq 4}|z_i|\leq |\x||\y|\leq \sqrt{B}$, if 
$x$ is represented by a vector $\x\in \ZZp^4$ and 
$y$ is represented by $\y\in \ZZp^4$.
We are left with counting $\z\in \ZZ_{\neq 0}^4$ such that $|\z|\leq \sqrt{B}$
and 
$c_1z_1^2+\cdots+c_4z_4^2=0$. The desired bound
$N(U,B)\ll_{\ve} B^{1+\ve}$ now follows directly from \eqref{eq:comtat}.

\medskip
\noindent
\underline{The case $(a,b)=(2,3)$.}
It follows from 
 Lemma~\ref{lem:CSab} that 
 \begin{equation}\label{eq:final1}
		N(U,\X,\Y)\ll_{\ve} X^\ve Y^\ve \min \left\{ X^{2}Y^{3}, X^3Y^2\right\}
		\end{equation}
		and (on 
		interchanging the role of $\X$ and $\Y$) it follows from Corollary~\ref{cor:a3}
	\begin{align}\label{eq:final4}
		N(U,\X,\Y)&\ll_{\ve} Y^{\ve}
		\left(Y^{\frac{12}{7}}X^{4-\frac{8}{21}}
		+Y^{\frac{43}{28}}X^{4-\frac{5}{42}}
		+YX^4 \right).
		\end{align}
We now combine these in \eqref{eq:dog} to prove the upper bound $N(U,B)\ll_\ve B^{\frac{44}{25}+\ve}$, noting that we may proceed under the assumption that 	 $X,Y\geq 1$ and $X^2Y\leq B$.
In particular, 
the factors $X^\varepsilon$ and
$Y^\varepsilon$ can be absorbed into $B^\varepsilon$.
We split the range according to the size of $X$, supposing first that
\[
X\leq B^{\frac{6}{25}}.
\]
Using \eqref{eq:final4} and the inequality $Y\leq B/X^2$, we see that 
\begin{align*}
Y^{\frac{12}{7}}X^{4-\frac{8}{21}}
&\leq \left(\frac{B}{X^2}\right)^{\frac{12}{7}}X^{\frac{76}{21}}
 = B^{\frac{12}{7}}X^{\frac{4}{21}}
 \leq B^{\frac{44}{25}},\\
Y^{\frac{43}{28}}X^{4-\frac{5}{42}}
&\leq \left(\frac{B}{X^2}\right)^{\frac{43}{28}}X^{\frac{163}{42}}
 = B^{\frac{43}{28}}X^{\frac{17}{21}}
\leq B^{\frac{173}{100}}
 \leq B^{\frac{44}{25}},\\
YX^4
&\leq X^4\frac{B}{X^2}
 = BX^2
 \leq 
 B^{\frac{37}{25}}
 \leq B^{\frac{44}{25}}.
\end{align*}
Hence this range makes a satisfactory overall contribution to $N(U,B)$. 
Now suppose that
\[
X\geq B^{\frac{6}{25}}.
\]
Using \eqref{eq:final1} and $Y\leq B/X^2$, we get the overall contribution 
\[
\ll_\varepsilon B^\varepsilon X^3Y^2
\leq B^\varepsilon X^3\left(\frac{B}{X^2}\right)^2
= B^{2+\varepsilon}X^{-1}
\leq B^{2-\frac{6}{25}+\varepsilon}
= B^{\frac{44}{25}+\varepsilon}
\]
to $N(U,B)$, which finally proves the claimed bound.

\medskip
\noindent
\underline{The case $(a,b)=(3,1)$.}
It follows from Lemma~\ref{cor:a3} that 
\[
N(U,\X,\Y)\ll_{\ve} X^{\ve}
\left(X^{\frac{12}{7}}Y^{4-\frac{4}{21}}
+X^{\frac{43}{28}}Y^{4-\frac{5}{84}}
+XY^4 \right).
\]
In this case \eqref{eq:height} yields $XY^3\ll B$. Hence it follows from \eqref{eq:dog} that 
\begin{align*}
N(U,B)
&\ll_{\ve} B^{\ve}
\left(B^{\frac{12}{7}}Y^{4-\frac{4}{21}-\frac{36}{7}}
+B^{\frac{43}{28}}Y^{4-\frac{5}{84}-\frac{129}{28}}
+BY \right)\\
&\ll_{\ve} B^{\ve}
\left(B^{\frac{12}{7}}Y^{-\frac{4}{3}}
+B^{\frac{43}{28}}Y^{-\frac{2}{3}}
+B^{\frac{4}{3}} \right)\\
&\ll_{\ve} B^{\frac{12}{7}+\ve},
\end{align*}
since $1\leq Y^3\ll B$. This is satisfactory for the theorem. 

\medskip
\noindent
\underline{The case $(a,b)=(3,3)$.}
Consider the surface 
$S\subset \PP^3$ defined by $c_1z_1^3+\cdots+c_4z_4^3=0$. 
Lines on this surface are contained in the subvarieties $c_{i}z_{i}^3+c_{j}z_{j}^3=c_{k}z_{k}^3+c_{\ell}z_{\ell}^3=0$, for $\{i,j,k,\ell\}=\{1,2,3,4\}$. 
Writing $z_i=x_i y_i$, as in the treatment of the case $(a,b)=(2,2)$, and noting that we are not counting points on $W_3$, it follows that we can exclude points on lines in $S$. 
The claimed bound $N(U,B)\ll_{\ve} B^{\frac{12}{7}+\ve}$ now follows directly from Salberger~\cite[Theorem 16.1]{salcubic}.
\end{proof}

\section{Application to some biprojective fourfolds}\label{sec:lin}
In this section, we prove Theorem~\ref{thm:aplllinear}. 
Let 
\begin{equation}\label{eq:graz}
F(\x,\y)=\ell_1(\y)x_1^2+\ell_2(\y)x_2^2+\ell_3(\y)x_3^2+\ell_4(\y)x_4^2,
\end{equation}
for nonzero linear forms 
$\ell_1,\ell_2,\ell_3,\ell_4\in \ZZ[y_1,y_2,y_3]$ such that any three of them are linearly independent. This is the polynomial defining the smooth fourfold 
 $Z'\subset \PP^3\times \PP^2$ 
considered in 
Theorem~\ref{thm:aplllinear}. Note that we can rewrite the defining polynomial 
as 
\begin{equation}\label{eq:graz'}
F(\x,\y)=q_1(\x)y_1+q_2(\x)y_2+q_3(\x)y_3,
\end{equation} 
for diagonal quadratic forms $q_1,q_2,q_3\in \ZZ[x_1,\dots,x_4]$. 
Moreover, if $C$ is the $3\times 4$ coefficient matrix of $q_1,q_2,q_3$, this matches 
the $3\times 4$ matrix whose columns are the coefficient vectors of $\ell_1,\dots,\ell_4$. Thus $\rank C=3$ under our assumptions on 
$\ell_1,\dots,\ell_4$.

\subsection{Technical results}

Recall the expression 
\eqref{eq:graz},
for linear forms 
$\ell_1,\dots ,\ell_4$ 
such that any three of them are linearly independent. 
Given $\y=(y_1,y_2,y_3)\in \ZZ_{\mathrm{prim}}^3$, let
\[g_{\y}=\gcd\left(\ell_1(\y),\dots,
\ell_4(\y)
\right),
\]
 and put $A_i(\y)=\ell_i(\y)/g_{\y}$, for $1\leq i\leq 4$. 
 Finally, let 
 \[T_{\y}=\gcd_{1\leq i<j\leq 4}(A_i(\y)A_j(\y)).\]
We will need the following result.

\begin{lemma}\label{lem:sumelligen}
 Let $\ve>0$ and let $a,b,c,d\geq 0$ be such that
 \[
 a+b,\ a+c,\ a+d,\ b+c,\ b+d,\ c+d\leq 1.
 \]
 Then
 \[
 \sum_{\substack{\y\in \ZZ_{\mathrm{prim}}^3\\
               \abs{\y}\leq Y\\
                   \prod_{i=1}^4 \ell_i(\y)\neq0}}
 \frac{T_{\y}^{\frac12}}
 {\abs{A_1(\y)}^a\abs{A_2(\y)}^b
  \abs{A_3(\y)}^c\abs{A_4(\y)}^d}
 \ll_{\ve, F}Y^{3-a-b-c-d+\ve}.
 \]
\end{lemma}

\begin{proof}
Let
\[
S_{a,b,c,d}=
 \sum_{\substack{\y\in \ZZ_{\mathrm{prim}}^3\\
                   \abs{\y}\leq Y\\
                   \prod_{i=1}^4 \ell_i(\y)\neq0}}
 \frac{T_{\y}^{\frac12}}
 {\abs{A_1(\y)}^a\abs{A_2(\y)}^b
  \abs{A_3(\y)}^c\abs{A_4(\y)}^d}
\]
be the quantity we wish to estimate.
Fix  $(\lambda_1,\lambda_2,\lambda_3,\lambda_4)\in
 \ZZ_{\mathrm{prim}}^4$  such that \[
 \lambda_1\ell_1+\cdots+\lambda_4\ell_4=0.
 \] 
 Since any three of $\ell_1,\dots,\ell_4$ are assumed to be linearly independent, we may proceed under the assumption that 
 $\lambda_1\cdots \lambda_4\neq 0$.
 We begin by proving that 
\[
g_{\y}\ll_F 1,
\]
for $\y\in \ZZ_{\mathrm{prim}}^3$.
To see this, choose three linearly
independent forms $\ell_1,\ell_2,\ell_3$ and let $L_0$ be their $3\times 3$ coefficient matrix, with 
$D=\det L_0\neq 0$. In particular, we have 
\[
\adj(L_0)(\ell_1(\y),\ell_2(\y),\ell_3(\y))^t=D\y.
\]
If $g_\y\mid (\ell_1(\y),\ell_2(\y),\ell_3(\y))$ then clearly $g_\y\mid Dy_i$ for $1\leq i\leq 3$, whence 
$g_\y\mid D$, since $\y$ is assumed to be primitive. This establishes the claim. 
Next, if $p^\nu\mid T_{\y}$ then exactly three of $A_1(\y),\dots,A_4(\y)$ are divisible by $p^\nu$.
But then it follows from the relation $\lambda_1\ell_1+\cdots+\lambda_4\ell_4=0$ that 
$\lambda_1A_1(\y)+\cdots+\lambda_4A_4(\y)=0$, whence 
 there exists a permutation $\{i,j,k,l\}=\{1,2,3,4\}$ such that 
 $p^\nu\mid \gcd(A_i(\y),A_j(\y),A_k(\y),\lambda_l)$.
It therefore follows that 
$T_{\y}\mid \lambda_1\cdots \lambda_4$, 
since $\lambda_1\cdots \lambda_4\neq 0$, 
whence also
\[
T_{\y} \ll_F 1.
\]

Our work so far shows that 
\begin{align*}
S_{a,b,c,d}&\ll_F
 \sum_{\substack{\y\in \ZZ_{\mathrm{prim}}^3\\
                  \abs{\y}\leq Y\\
                  \prod_{i=1}^4\ell_i(\y)\neq0}}
 \frac{1}
 {\abs{\ell_1(\y)}^a\abs{\ell_2(\y)}^b
  \abs{\ell_3(\y)}^c\abs{\ell_4(\y)}^d}
\\
&\ll_F
 \sum_{\substack{\y\in \ZZ_{\mathrm{prim}}^3\\
                  \abs{\y}\leq Y\\
                  \prod_{i=1}^4\ell_i(\y)\neq0}}
 \frac{1}
 {\abs{\ell_1(\y)}^a\abs{\ell_2(\y)}^b
  \abs{\ell_3(\y)}^c
  \abs{\lambda_1\ell_1(\y)+\lambda_2\ell_2(\y)
       +\lambda_3\ell_3(\y)}^d}.
\end{align*}
The map $\y\mapsto(\ell_1(\y),\ell_2(\y),\ell_3(\y))$ is injective,
since 
$\ell_1,\ell_2,\ell_3$ are linearly independent.
Moreover, the image for $\y\in\ZZ_{\mathrm{prim}}^3$ such that
$\abs{\y}\leq Y$ is contained in
$[-C_FY,C_FY]^3$,
for an appropriate constant $C_F>0$.
Thus 
\begin{align*}
S_{a,b,c,d}
&\ll_F
 \sum_{\substack{m_1,m_2,m_3\in\ZZ\\
                  0<\abs{m_i}\ll_FY\\
                  \lambda_1m_1+\lambda_2m_2+\lambda_3m_3\neq0}}
 \frac{1}
 {\abs{m_1}^a\abs{m_2}^b\abs{m_3}^c
  \abs{\lambda_1m_1+\lambda_2m_2+\lambda_3m_3}^d}.
\end{align*}
Clearly 
\[
 \frac{1}{\abs{r}^{\alpha}\abs{qr+s}^{\beta}}
 \leq
 \frac{1}{\abs{qr+s}^{\alpha+\beta}}
 +\frac{1}{\abs{r}^{\alpha+\beta}},
\]
for any $\alpha,\beta\geq0$ and $r(qr+s)\neq0$.
Hence 
\begin{align*}
S_{a,b,c,d}
\ll_F~&
 \sum_{\substack{m_1,m_2,m_3\in\ZZ\\
                  0<\abs{m_i}\ll_FY}}
 \frac{1}{\abs{m_1}^{a+d}\abs{m_2}^b\abs{m_3}^c}\\
&+
 \sum_{\substack{m_1,m_2,m_3\in\ZZ\\
                  0<\abs{m_i}\ll_FY\\
                  \lambda_1m_1+\lambda_2m_2+\lambda_3m_3\neq0}}
 \frac{1}
 {\abs{m_2}^b\abs{m_3}^c
 \abs{\lambda_1m_1+\lambda_2m_2+\lambda_3m_3}^{a+d}}.
\end{align*}
The map $m_1\mapsto \lambda_1m_1+\lambda_2m_2+\lambda_3m_3$ is clearly injective, for fixed $m_2,m_3$, whence
\begin{align*}
S_{a,b,c,d}
&\ll_F
 \sum_{\substack{m_1,m_2,m_3\in\ZZ\\
                  0<\abs{m_i}\ll_FY}}
 \frac{1}{\abs{m_1}^{a+d}\abs{m_2}^b\abs{m_3}^c}
\ll_{\ve,F}Y^{3-a-b-c-d+\ve}. \qedhere
\end{align*}
\end{proof}

Recall the introduction of the quadratic forms  $q_1,q_2,q_3$ in 
\eqref{eq:graz'}.
Writing $\q(\x)=(q_1(\x),q_2(\x),q_3(\x))$, let $\x\in \ZZ_{\prim}^4$ be such that $\q(\x)\neq (0,0,0)$. Let 
\[
g_\x=\gcd(q_1(\x),q_2(\x),q_3(\x))
\]
 and put 
$h_i(\x)=q_i(\x)/g_\x$, for $1\leq i\leq 3$. 
Let $\h(\x)=(h_1(\x),h_2(\x),h_3(\x))\in \ZZ_{\prim}^3$. 
We will need the following result. 

\begin{lemma}\label{lem:sumq}
Let $\ve>0$. Then
\[
 \sum_{\substack{\x\in \ZZ_{\prim}^4\\\abs{\x}\leq X\\
 \q(\x) \neq (0,0,0)}}\frac{1}{|\h(\x)|}\ll_{\ve,F} X^{2+\ve}.
 \]
\end{lemma}

\begin{proof}
Let $C$ be the $3\times 4$ coefficient matrix of $q_1,q_2,q_3$, which we have already seen has rank $3$. The first three columns form a nonsingular matrix $C_0$. Write
$D=\det C_0$ and set
\[
(r_1(\x),r_2(\x),r_3(\x))^t
 =\adj(C_0)(q_1(\x),q_2(\x),q_3(\x))^t.
\]
Then
\[
r_i(\x)=Dx_i^2+a_ix_4^2
\]
for some fixed $a_i\in\ZZ$. 
Letting $\r(\x)=(r_1(\x),r_2(\x),r_3(\x))$, 
it follows from the identities $\mathbf r(\x)^t=\adj(C_0)\mathbf q(\x)^t$
and 
$\mathbf q(\x)^t=D^{-1}C_0\mathbf r(\x)^t$
that 
\[
\gcd (r_1(\x),r_2(\x), r_3(\x))\ll_F
\gcd (q_1(\x),q_2(\x),q_3(\x))\ll_F \gcd (r_1(\x),r_2(\x), r_3(\x))
\]
and 
\[
 |\r(\x)|\ll_F
|\q(\x)| \ll_F |\r(\x)|.
\]
Hence it is enough
to prove
\[
\sum_{\substack{|\x|\leq X\\\mathbf r(\x)\neq (0,0,0)}} 
 \frac{\gcd (r_1(\x),r_2(\x), r_3(\x))}{|\r(\x)|}
 \ll_{\ve,F}X^{2+\ve},
\]
where we have also dropped the primitivity condition.

We shall use the elementary estimate
\begin{equation}\label{eq:quadratic-congruence}
 \#\{u\in I\cap\ZZ:d\mid Du^2+b\}
 \ll_{\ve,F}d^\ve\left(1+\frac{|I|}{\sqrt d}\right)
\end{equation}
for any interval $I$ and any $b\in \ZZ$ and $d\in\NN$. To see this, fix one
solution $u_0$ and put $d'=d/\gcd(d,D)$. Every other solution satisfies
\[
d'\mid(u-u_0)(u+u_0).
\]
Distributing the prime-power factors of $d'$ between the two factors
places the solutions in $O_\ve(d^\ve)$ arithmetic progressions, each
of modulus $\gg_F \sqrt d$.

Now let $M$ run over powers of $2$ with $1\leq M\ll_F X^2$. Fix $x_4$ and write
\[
Q_i(u)=Du^2+a_ix_4^2,
\qquad
I_i=\{u\in[-X,X]:|Q_i(u)|\leq 2M\}.
\]
Each $I_i$ is a union of $O_F(1)$ intervals of total length $O_F(\sqrt{M})$.
We claim that
\begin{equation}\label{eq:gcd-average}
 \sum_{\substack{u_i\in I_i\cap\ZZ\\
 (Q_1(u_1),Q_2(u_2),Q_3(u_3))\neq (0,0,0)}}
 \hspace{-0.3cm}
 \gcd\left(Q_1(u_1),Q_2(u_2),Q_3(u_3)\right)
 \ll_{\ve,F}M^{\frac{3}{2}+\ve}.
\end{equation}
Indeed, using
\[
\gcd(m_1,m_2,m_3)=\sum_{d\mid m_1,m_2,m_3}\varphi(d)
\]
and choosing one nonzero coordinate, the left-hand side of
\eqref{eq:gcd-average} is at most
\[
\sum_{j=1}^3
\sum_{\substack{u_j\in I_j\cap \ZZ\\Q_j(u_j)\neq0}}
\sum_{d\mid Q_j(u_j)}
 \varphi(d)
 \prod_{i\neq j}
 \#\{u_i\in I_i\cap \ZZ: d\mid Q_i(u_i)\}.
\]
By~\eqref{eq:quadratic-congruence}, this is
\[
\ll_{\ve,F}
M^\ve
\sum_{j=1}^3
\sum_{u_j\in I_j\cap \ZZ}
\sum_{d\mid Q_j(u_j)}
 d\left(1+\frac{\sqrt{M}}{\sqrt d}\right)^2.
\]
Since $d\leq |Q_j(u_j)|\leq 2M$, it follows that 
\[
d\left(1+\frac{\sqrt{M}}{\sqrt d}\right)^2
\ll M.
\]
Since each $I_j$ contains $O_F(\sqrt{M})$ integers, the divisor bound gives
\eqref{eq:gcd-average}.

Consequently, the overall contribution from
$M\leq|\r(\x)|<2M$
is
\[
\ll_{\ve,F}
\frac1M\sum_{|x_4|\leq X}M^{\frac{3}{2}+\ve}
\ll_{\ve,F}XM^{\frac{1}{2}+\ve}.
\]
Summing over the dyadic values $M\ll_F X^2$ finally yields the statement of the lemma. 
\end{proof}

\subsection{Completion of the proof}

Let $Z'\subset \PP^3\times \PP^2$ be the smooth fourfold defined by 
$F(\x,\y)=0$, in which we recall 
the expressions \eqref{eq:graz} and \eqref{eq:graz'}.
Let 
\[
W'=\left\{(x,y)\in Z'(\QQ):
\q(\x)= (0,0,0) 
 \text{ or } \prod_{i=1}^4 \ell_i(\y)\in \QQ^{\times 2}
  \text{ or } \prod_{i=1}^4 \ell_i(\y) x_i\prod_{i=1}^3y_i=0\right\}.
\]
This is a thin set and we henceforth put $U'=Z'(\QQ)\setminus W'$.

We proceed by defining $N(U',\X,\Y)$ as in 
\eqref{eq:fox}. For $\X=(X_1,X_2,X_3,X_4)\in \RR_{>0}^4$ and $\Y=(Y_1,Y_2,Y_3)\in \RR_{>0}^3$, let 
	\[
	N(U',\X,\Y)=\#\left\{(\x,\y)\in \ZZ_{\prim}^4\times \ZZ_{\prim}^3: (x,y)\in U',~
	\x\sim \X,~ \y\sim \Y\right\},
	\]
	where $x=[\x]$ and $y=[\y]$ are the corresponding points in projective space and $U'=Z'(\QQ)\setminus W'$. The following two results are our primary estimates for 
	$N(U',\X,\Y)$.

\begin{lemma}\label{lem:boundquad}
Let $\ve>0$. Then
 \[
 N(U',\X,\Y)\ll_{\ve,F} X^{\frac 43+\ve}Y^3+X^{2+\ve}Y^{2+\ve}.
 \]
\end{lemma}
\begin{proof}
 Let $\y\in \ZZ_{\prim}^3$ such that $\prod_i\ell_i(\y)\not \in \QQ^{\times 2}$ and $\prod_i\ell_i(\y)\neq 0$. Let $Q_{\y}(\x)=0$ be the diagonal quadratic surface 
 occurring in \eqref{eq:graz}, for fixed $\y$. Our assumption on $\y$ ensures that the surface contains no rational lines. Let $M_\y=\max_i\{\abs{A_i(\y)}\}$ and $m_{\y}=\min_i\{\abs{A_i(\y)}\}$,
 where we recall that 
 \[
 A_i(\y)=\frac{\ell_i(\y)}{\gcd(\ell_1(\y),\dots,\ell_4(\y))}.
 \]
It follows from Theorem~\ref{prop:main} that 
\begin{align*}
N(Q_{\y},X)\ll_{\ve} ~&
		X^{\frac43+\ve}
		+
		\frac{X^{\frac43+\ve}T_{\y}^{\frac12}
			M_\y^{\frac32}}{\abs{A_1(\y)A_2(\y)A_3(\y)A_4(\y)}^{\frac12}}
+					\frac{X^{\frac53+\ve}T_{\y}^{\frac12}M_\y}{m_{\y}^{\frac16}\abs{A_1(\y)A_2(\y)A_3(\y)A_4(\y)}^{\frac13}} 
		\\&		+
		\frac{X^{2+\ve}T_{\y}^{\frac12}M_\y^{\frac12}}{m_{\y}^{\frac16}\abs{A_1(\y)A_2(\y)A_3(\y)A_4(\y)}^{\frac13}}.
\end{align*}
But then, on introducing the sum over $\y$, we see that 
 \begin{align*}
 N(U',\X,\Y)\ll_\ve~& 
		X^{\frac43+\ve}Y^3
+\Sigma_1+\Sigma_2+\Sigma_3,
\end{align*}
where
\begin{align*}
\Sigma_1&=
\sum_{\substack{\y\in \ZZ_{\prim}^3\\ \abs{\y}\ll Y\\ 
 \prod\ell_i(\y) \neq 0}}
		\frac{X^{\frac43+\ve}T_{\y}^{\frac12}
			M_\y^{\frac32}}{\abs{A_1(\y)A_2(\y)A_3(\y)A_4(\y)}^{\frac12}},\\
\Sigma_2&=
\sum_{\substack{\y\in \ZZ_{\prim}^3\\ \abs{\y}\ll Y\\ 
\prod\ell_i(\y) \neq 0}}
		\frac{X^{\frac53+\ve}T_{\y}^{\frac12}M_\y}{m_{\y}^{\frac16}\abs{A_1(\y)A_2(\y)A_3(\y)A_4(\y)}^{\frac13}},
		\end{align*}
		and 
\[
\Sigma_3=
\sum_{\substack{\y\in \ZZ_{\prim}^3\\ \abs{\y}\ll Y\\ 
 \prod\ell_i(\y) \neq 0}}
		 		\frac{X^{2+\ve}T_{\y}^{\frac12}M_\y^{\frac12}}{m_{\y}^{\frac16}\abs{A_1(\y)A_2(\y)A_3(\y)A_4(\y)}^{\frac13}}.
				\]
Noting that $M_\y\ll_F Y$, it follows from Lemma~\ref{lem:sumelligen} with $a=b=c=d=1/2$ 
that 
\[
\Sigma_1\ll_{\ve,F}X^{\frac 43+\ve}Y^{\frac 32}Y^{1+\ve}.
\]
Next, on noting that 
\[
\frac{1}{m_{\y}^{\frac16}}\leq \frac{1}{\abs{A_1(\y)}^{\frac 16}}+\frac{1}{\abs{A_2(\y)}^{\frac 16}}+\frac{1}{\abs{A_3(\y)}^{\frac 16}}+\frac{1}{\abs{A_4(\y)}^{\frac 16}},
\]
it follows from Lemma~\ref{lem:sumelligen} with three exponents equal to $1/3$ and one equal to $1/2$, that
\[
\Sigma_2\ll_{\ve,F} X^{\frac 53+\ve}Y^{1}Y^{\frac 32+\ve} \quad \text{ and }
\quad 
\Sigma_3
\ll_{\ve,F} X^{2+\ve}Y^{\frac 12}Y^{\frac 32+\ve}.
\]

Finally, on putting everything together, we obtain 
\begin{align*}
 N(U',\X,\Y)&\ll_{\ve,F} X^{\frac 43+\ve}Y^3+X^{\frac 43+\ve}Y^{\frac 52+\ve}+X^{\frac 53+\ve}Y^{\frac 52+\ve}+X^{2+\ve}Y^{2+\ve}
 \\ &\ll_{\ve,F}X^{\frac 43+\ve}Y^{2+\ve}\left(Y+X^{\frac 13}Y^{\frac 12}+X^{\frac 23}\right).
\end{align*}
This completes the proof, since
$X^{\frac 13}Y^{\frac 12}\ll Y+X^{\frac 23}$, 
which follows from the elementary inequality $2uv\leq u^2+v^2$.
\end{proof}

\begin{lemma}\label{lem:boundline}
Let $\ve>0$. Then 
 \[
 N(U',\X,\Y)\ll_{\ve,F} X^{2+\ve}Y^{2}+X^4.
 \]
\end{lemma}
\begin{proof}
We rewrite $F(\x,\y)$ as in \eqref{eq:graz'}, for 
quadratic forms $q_1,q_2,q_3\in \ZZ[x_1,\dots,x_4]$. 
Let $\x\in \ZZ_{\prim}^4$ be such that $(q_1(\x),q_2(\x),q_3(\x))$ is not the zero vector and 
recall the notation
$\h(\x)=(h_1(\x),h_2(\x),h_3(\x))\in \ZZ_{\prim}^3$, where
$h_i(\x)=q_i(\x)/\gcd(q_1(\x),q_2(\x),q_3(\x))$, for $1\leq i\leq 3$. 
The equation
\[
h_1(\x)y_1+h_2(\x)y_2+h_3(\x)y_3=0
\]
defines a line $L_{\x}\subset\mathbb P^2$.
It follows from~\cite[Lemma 2]{H-B} that 
 \[
 N(L_{\x},Y)\ll 1+ \frac{Y^{2}}{|\h(\x)|}.
 \]
Applying Lemma~\ref{lem:sumq}, we conclude that 
 \[
 N(U',\X,\Y)\ll \sum_{\substack{\x\in \ZZ_{\prim}^4\\ \abs{\x}\ll X\\ \h(\x)\neq (0,0,0)}}N(L_{\x},Y)\ll_{\ve,F} X^{2+\ve}Y^{2}+X^4. \qedhere
 \]
\end{proof}

\begin{proof}[Proof of Theorem~\ref{thm:aplllinear}]
 Splitting the intervals in dyadic boxes, we need to prove the bound for $N(U',\X,\Y)$ with $XY\leq B^{\frac 12}$. If $X\leq B^{\frac{9}{34}}$, then Lemma~\ref{lem:boundline} yields
 \[
 N(U',\X,\Y)\ll_{\ve,F} B^{\frac{18}{17}+\ve}.
 \]
 If $X\geq B^{\frac{9}{34}}$, on the other hand, then Lemma~\ref{lem:boundquad} yields
 \begin{align*}
 N(U',\X,\Y)\ll_{\ve,F} X^{\frac 43+\ve}Y^{3}+X^{2+\ve}Y^{2+\ve}
 &\ll_{\ve,F}
X^{\frac 43-3}B^{\frac{3}{2}+\ve}+B^{1+\ve}
\ll_{\ve,F} B^{\frac{18}{17}+\ve}. \qedhere
 \end{align*}
\end{proof}

\end{document}